\documentclass[reqno,12pt]{amsart}
\usepackage{epsfig}\usepackage{amsfonts,amsmath,amssymb,mathrsfs,verbatim,amsthm,amscd}
\def\stackreb#1#2{\ \mathrel{\mathop{#1}\limits_{#2}}}
\def\II{\hbox{{1}\kern-.25em\hbox{l}}} 
\newcommand{\beq}{\begin{equation}}\newcommand{\eeq}{\end{equation}}
\newcommand{\baq}{\begin{eqnarray}}\newcommand{\eaq}{\end{eqnarray}}

\newcommand{\R}{\mathbb R}

\newcommand{\T}{\mathbb T}

\newcommand{\eg}{\Gamma}
\newtheorem{theorem}{Theorem}
\newtheorem{conjecture}{Conjecture}

\begin{document}

\title[Rarefied elliptic hypergeometric functions]
{Rarefied elliptic hypergeometric functions}

\author{V.\,P. Spiridonov}
 \address{Bogoliubov Laboratory of Theoretical Physics,
JINR, Dubna, Moscow region, 141980 Russia and
 St. Petersburg Department of the Steklov Mathematical Institute
of Russian Academy of Sciences,
Fontanka 27, St. Petersburg, 191023 Russia
  {\sc e-mail:  spiridon@theor.jinr.ru}}

\begin{abstract}
Two exact evaluation formulae for multiple rarefied elliptic beta
integrals related to the simplest lens space are proved. They generalize
evaluations of the type I and II elliptic beta integrals attached to the
root system $C_n$. In a special $n=1$ case, the simplest $p\to 0$ limit is
shown to lead to a new class of $q$-hypergeometric identities. Symmetries of
a rarefied elliptic analogue of the Euler-Gauss hypergeometric function are described
and the respective generalization of the hypergeometric equation is constructed.
Some extensions of the latter function to $C_n$ and $A_n$ root systems and
corresponding symmetry transformations are considered. An application of the
rarefied type II $C_n$ elliptic hypergeometric function to some eigenvalue
problems is briefly discussed.
\end{abstract}

\maketitle


\tableofcontents

\hfill{\em  To the memory of G.\,M. Vereshkov, my teacher}

\medskip

\section{Introduction}

Hypergeometric functions are central objects in the theory of special
functions \cite{aar}. Elliptic functions (i.e., meromorphic
doubly-periodic functions) form another key family from this world.
Nowadays it is known that these two classical sets of functions
are deeply tied to each other.
Frenkel and Turaev \cite{FT} investigated elliptic functions that
appeared as solutions of the Yang-Baxter equation in \cite{djkmo}
and have shown that they have the form resembling hypergeometric
series and obey similar properties. This has led to the discovery
of remarkable terminating elliptic hypergeometric series summation
and transformation formulae (infinite series of such
type do not converge) generalizing the corresponding Jackson sum
and Bailey transformation \cite{aar,GR}. The biorthogonal functions
expressed in terms of such series have been constructed in \cite{SZ}
(discrete measure) and \cite{spi:theta2} (continuous measure).

The genuine elliptic hypergeometric functions
which are transcendental over the field of elliptic functions were
discovered in \cite{spi:umn}. They are determined by a specific
class of integrals whose integrands satisfy linear $q$-difference equations
with $p$-elliptic coefficients, and this property can be used for a general
definition of such functions \cite{spi:theta2}. Elliptic hypergeometric
integrals can be reduced to elliptic hypergeometric series (particular
elliptic functions) through residues calculus, to general class of
$q$-hypergeometric functions (by taking a simple limit $p\to 0$
\cite{spi:umn}, or by more complicated degenerations \cite{rai:limits})
and to ordinary hypergeometric functions.
Unification of elliptic and hypergeometric insights elucidated various
previously known properties of the corresponding functions. For instance,
it explained the origin of hypergeometric notions of well-poisedness,
very-well-poisedness and balancing in terms of the ellipticity conditions
\cite{spi:umnrev}.
The unique nature of the most interesting elliptic hypergeometric
functions is established by their symmetries associated with two
independent elliptic curves and two independent root systems, or compact
Lie groups (one attached to the Haar measure defining multiple integrals
and another one living in the space of free parameters of these functions).
The basics of the theory of elliptic hypergeometric functions is surveyed
in \cite{spi:umnrev}, a more recent review is given in \cite{rosengren}.

The elliptic beta integral \cite{spi:umn} is until now the only known computable
integral among univariate elliptic hypergeometric integrals. Its evaluation
formula (see identity \eqref{ebeta}) can be interpreted as an elliptic analogue
of Newton's binomial theorem. Moreover, this integral represents a top known
generalization of Euler's beta function and serves as the biorthogonality
measure for particular elliptic functions (actually, the product of two elliptic
functions with different nomes \cite{spi:theta2}) forming the most general
set of special functions extending the Jacobi and Askey-Wilson
polynomials \cite{aar}, etc.
It has found remarkable applications in theoretical physics,
the first one being in quantum mechanical eigenvalue problems
\cite{spi:thesis,spi:tmf2007}.
The most important physical interpretation of formula \eqref{ebeta}
was discovered by Dolan and Osborn in quantum field theory \cite{DO}
-- it proves the equality of superconformal indices of two nontrivial four
dimensional supersymmetric models connected by the Seiberg duality.
Currently this gives the most rigorous mathematical confirmation of
the confinement phenomenon.

The next important representative of univariate elliptic hypergeometric
integrals is an elliptic analogue of the Euler-Gauss hypergeometric function.
It contains two more free parameters than the elliptic beta integral
and satisfies an elliptic hypergeometric equation \cite{spi:thesis,spi:tmf2007}.
It appeared for the first time in \cite{spi:theta2} together with a
nontrivial symmetry transformation related to the exceptional root
system $E_7$ (an identification with the corresponding Weyl group
action was established in \cite{rai:trans}).

Another line of generalizations of the elliptic beta integral
considers multiple integrals. Here, the very first multiple
elliptic beta integrals were defined by van Diejen and the present
author \cite{vDS1,vDS2} in relation to the root system $C_n$.
The integrals defined in \cite{vDS2}
are referred to as of type I, and those of \cite{vDS1}
as of type II. The latter integral is a generalization of
the Selberg integral \cite{Selberg}, which follows from its simple
$p\to0$ reduction to a Gustafson integral with previously known reductions
\cite{gust1,gust2}. The importance of Selberg's multiple beta integral
and its generalizations is surveyed in \cite{for-war}.
The classification of integrals as type I or II is inspired by differences
in the methods used for proving corresponding exact evaluation formulae
and in the number of free parameters contained in them (which depends on the rank
of the root system for type I and is fixed for type II integrals). The type II
integral evaluation was proven to follow from the type I integral identity as a result
of simple, purely algebraic considerations \cite{vDS2}. However, the type I
integral evaluation was not completely proven in \cite{vDS2}, it depended on
a vanishing condition of a certain integral. The first complete proof of this
integration formula was given by Rains in \cite{rai:trans}, and shortly
thereafter the present author found an elementary proof of the type I integral
evaluations \cite{spi:short} (this method was used also for considering
other integrals in \cite{SW} and we employ it here as well).
Nowadays, the number of known elliptic hypergeometric integrals
admitting (proven or conjectural) either exact evaluation or a nontrivial
symmetry transformation described by the Weyl groups of various root systems is
very large, see, e.g. \cite{bult:trafo,bult:quadr,rai:trans,spi:umnrev,SV1,SV2,SW}.
In this paper we crucially follow the logic established in \cite{vDS1,vDS2} and
use it for a derivation of elliptic hypergeometric identities of a new type.

The discovery of relations between superconformal indices
and elliptic hypergeometric functions \cite{DO} has attracted much attention.
In addition to its systematic consideration in \cite{SV1,SV2}, which resulted in
the formulation of very many new mathematical conjectures and the discovery of
new physical Seiberg dualities, there have been other important developments.
For instance, such integrals emerged in two-dimensional topological field
theories \cite{gprr},  their properties describe symmetry enhancement phenomena
\cite{DG}, the elliptic Fourier transform introduced in \cite{spi:bailey2}
in the rank 1 case and extended to arbitrary rank root systems in
\cite{SW} plays an
important role in the discussion of five dimensional duality questions \cite{GK}, etc.
For a recent survey of this subject, see \cite{RR2016}.

In this setting, the standard elliptic hypergeometric integrals are related to the
Hopf manifold $S^1\times S^3$, which plays a role of compact space-time
for the corresponding four dimensional superconformal field theories.
However, this is only one of many admissible four dimensional manifolds for which
one can compute superconformal indices. The next level of topological
complication is related to the replacement of the $S^3$-factor by the
lens space. It was considered first in \cite{ben}, where an analogue
of the elliptic gamma function for the simplest lens space was introduced.
Some further essential developments of this subject can be found in \cite{RW}.

Recently, Kels \cite{kels} proposed an extension of the univariate elliptic beta
integral associated with the simplest lens space. It involves
some additional discrete variables and a replacement of the single
integration by a finite sum of integrations. Earlier examples
of similar sums of simpler integrals can be found in \cite{BMS,kels0,N}.
In particular, they emerge already in the representation theory of $\textrm{SL}(2,\mathbb{C})$
group \cite{N}. In this work we consider elliptic hypergeometric functions associated
with the lens space by a different method and confirm
the result derived in \cite{kels}. Moreover, we find another similar extension of
formula  \eqref{ebeta} and show that it has a highly nontrivial $p\to 0$
limit leading to a new type of $q$-hypergeometric identities.
Furthermore, we propose two explicit multiple hybrid sum-integrals
generalizing the type I and II elliptic beta integrals of \cite{vDS1,vDS2}.
It will be shown that the general class of such functions, which we
call ``rarefied elliptic hypergeometric functions", matches with the general
definition of elliptic hypergeometric functions of \cite{spi:theta2}
when it is applied to the case of sums of integrals.
As follows from their structure, these functions should be considered as
$2n$-variate ``sum-integral" objects. One set of $n$ variables is discrete
and defines the $n$-tuple summation and another set of $n$ continuous variables
defines the $n$-tuple integration.

We define a rarefied elliptic analogue of the Euler-Gauss hypergeometric function
and construct the corresponding $W(E_7)$ symmetry transformations and elliptic
hypergeometric equation. The type II extension of the latter function to the root
systems $C_n$ is proposed and an extension of Rains' transformation \cite{rai:trans}
is conjectured. An application of this function to the eigenvalue problem for a
particular finite-difference operator is briefly discussed. We consider also symmetry
transformations for the rarefied multiple elliptic hypergeometric functions of type I
on the root systems $C_n$ and $A_n$. In the concluding section we outline some
prospects for further development of the derived results.

\section{The elliptic beta integral}

In this section we describe particular elliptic hypergeometric integrals
introduced in \cite{vDS1,vDS2,spi:umn}.
For  $p\in\mathbb{C}$, $|p|<1$, we define the infinite product
$$
(z;p)_\infty=\prod_{j=0}^\infty(1-zp^j), \quad z\in \mathbb{C}.
$$
The theta function
\beq
\theta(z;p)=(z;p)_\infty(pz^{-1};p)_\infty, \quad z\in\mathbb{C}^*,
\label{thetafn}\eeq
obeys the following symmetry properties
$$
\theta(x^{-1};p)=\theta(px;p)=-x^{-1}\theta(x;p).
$$
We shall need the general quasiperiodicity relation
$$
\theta(p^kz;p)=(-z)^{-k}p^{-\frac{k(k-1)}{2}}\theta(z;p), \quad k\in\mathbb{Z}.
$$
The ``addition law" for theta functions has the form
\beq
\theta(xw^{\pm 1},yz^{\pm 1};p) -\theta(xz^{\pm 1},yw^{\pm 1};p)
=yw^{-1}\theta(xy^{\pm 1},wz^{\pm 1};p),
\label{addition}\eeq
where $x,y,w,z\in\mathbb{C}^*$.
We use the convention
$$
\theta(x_1,\ldots,x_k;p)=\prod_{j=1}^k\theta(x_j;p), \qquad
\theta(tx^{\pm1};p)=\theta(tx,tx^{-1};p).
$$

For arbitrary $q\in \mathbb{C}$ and $n\in\mathbb{Z}$,  the
elliptic Pochhammer symbol is defined as
$$
\theta(x;p|q)_n:=\begin{cases}
 \prod_{j=0}^{n-1}\theta(xq^j;p), & \text{for $n>0$} \\
\prod_{j=1}^{-n} \theta(xq^{-j};p)^{-1}, & \text{for $n<0$}
\end{cases}
$$
and $\theta(x;p|q)_0=1$.

The first order $q$-difference equation
\beq
f(qz;p,q)=\theta(z;p)f(z;p,q), \qquad q\in \mathbb{C}^*,
\label{egeqn}\eeq
has a particular solution
\begin{equation}
f(z;p,q)=\Gamma(z;p,q):=
\prod_{j,k=0}^\infty\frac{1-z^{-1}p^{j+1}q^{k+1}}{1-zp^{j}q^{k}},
\qquad |p|, |q|<1,\quad z\in\mathbb{C}^*,
\label{eg}\end{equation}
called the (standard) elliptic gamma function. Note that equation
\eqref{egeqn} does not require $|q|<1$, whereas solution \eqref{eg}
is restricted to this domain.

The function $\Gamma(z;p,q)$ is a building block for
elliptic hypergeometric integrals with an interesting history.
The problem of generalizing Euler's gamma function to a similar function
``of the second order" was considered by Alexeevski long ago \cite{A}.
Later Barnes derived similar results, but he went further and defined
multiple gamma functions of arbitrary order \cite{bar:multiple}.
Jackson \cite{jac:basic} considered this problem from a different angle
and investigated the well-known $q$-gamma function and defined the double
$(p, q)$-gamma function. In the same work he constructed the elliptic
gamma function with equal periods $p=q$ and indicated how to construct
the general case, however, his results did not attract the deserved attention.
This function was implicitly discovered also by Baxter in \cite{Bax},
since the partition function of the eight-vertex model is given by
a particular combination of four such functions (see \cite{fel-var:elliptic}
for an explicit relation). More recently function \eqref{eg} was
considered by Ruijsenaars in \cite{rui:first}, where the term ``elliptic
gamma function" was introduced. A systematic investigation of this function
was performed by Felder and Varchenko \cite{fel-var:elliptic} who discovered
its $\textrm{SL}(3,\mathbb{Z})$ symmetry transformations and gave a
cohomological interpretation. In \cite{spi:theta2} the author
constructed the modified elliptic gamma function, which gives
a solution of equation \eqref{egeqn} in the regime $|q|=1$
(it is meromorphic in $\log z$, not $z$) and has nice physical
applications \cite{SV3}. In \cite{FR}, Friedman
and Ruijsenaars explicitly expressed the
standard elliptic gamma function \eqref{eg} as a particular combination of
four Barnes gamma functions of the third order
(such a relation was suggested also earlier in \cite{spi:theta2} up to
the exponential of a Bernoulli polynomial factor).

The elliptic gamma function \eqref{eg} has the following properties
$$
\Gamma(z;p,q) =\Gamma(z;q,p), \qquad
$$
$$
\Gamma(qz;p,q)=\theta(z;p)\Gamma(z;p,q), \quad
\Gamma(pz;p,q)=\theta(z;q)\Gamma(z;p,q),
$$
\beq\label{inv1}
\Gamma(z;p,q)\Gamma(\textstyle{\frac{pq}{z}};p,q) =1,
\eeq
$$
 \Gamma(\sqrt{pq};p,q)=1,
$$
 \beq\label{resid}
\stackreb{\lim}{z\to 1}(1-z) \Gamma(z;p,q)=\frac{1}{(p;p)_\infty(q;q)_\infty}.
 \eeq
Its poles and zeros form the double base geometric progressions
\beq
z_{\mathrm{poles}}=p^{-j}q^{-k},\quad
z_{\mathrm{zeros}}=p^{j+1}q^{k+1},\quad j,k\in\mathbb{Z}_{\geq 0}.
\label{zp_egf}\eeq

The elliptic Pochhammer symbol can be written in the form
$$
\theta(z;p|q)_m=\frac{\Gamma(zq^m;p,q)}{\Gamma(z;p,q)}, \quad m\in\mathbb{Z}.
$$

We also define the elliptic gamma function of the second order
$$
\Gamma(z;p,q,t)=\prod_{j,k,l=0}^\infty (1-zp^jq^kt^l)
(1-z^{-1} p^{j+1}q^{k+1}t^{l+1}), \quad |t|,|p|,|q|<1,\; z\in\mathbb{C}^*.
$$
It satisfies the equation
\beq
\Gamma(qz;p,q,t)=\Gamma(z;p,t)\Gamma(z;p,q,t)
\label{2ndordereqn}\eeq
and its partners obtained by permutation of the bases $p,q,t$
and the inversion relation
$$
 \Gamma(pqtz;p,q,t)=\Gamma(z^{-1};p,q,t).
$$

The elliptic beta integral evaluation formula, which serves as a basis
for the whole general theory of elliptic hypergeometric integrals,
has the following form \cite{spi:umn}.

\begin{theorem} Let $t_1, \dots ,t_6,p,q \in \mathbb{C}^*$ such that
$|t_a|,|p|,|q| <1$ and $\prod_{a=1}^6 t_a=pq$. Then
\begin{equation} \label{ebeta}
\frac{(p;p)_\infty (q;q)_\infty}{2} \int_{\mathbb{T}} \frac{\prod_{a=1}^6
\Gamma(t_a z^{\pm 1} ;p,q)}{\Gamma(z^{\pm 2};p,q)} \frac{dz}{2 \pi \textup{i} z}
= \prod_{1 \leq a < b \leq 6} \Gamma(t_a t_b;p,q),
\end{equation}
where $\mathbb T$ is the positively oriented unit circle.
\end{theorem}

Here we use the compact notation
\begin{eqnarray*} &&
\Gamma(t_1,\ldots,t_n;p,q):=\Gamma(t_1;p,q)\ldots\Gamma(t_n;p,q), \quad
\Gamma(tz^{\pm k};p,q):=\Gamma(tz^k;p,q)\Gamma(tz^{-k};p,q), \quad
\\ &&
\Gamma(z^{\pm 1}w^{\pm1};p,q):=\Gamma(zw;p,q)\Gamma(z^{-1}w;p,q)
\Gamma(zw^{-1};p,q)\Gamma(z^{-1}w^{-1};p,q).
\end{eqnarray*}

If one substitutes in \eqref{ebeta} $t_6=pq/\prod_{a=1}^5 t_a$,
uses the inversion formula \eqref{inv1}, and takes the limit $p\to 0$
for fixed $t_1,\ldots, t_5$ and $q$, then one obtains the Rahman $q$-beta integral
\cite{Rahman} (see formula \eqref{rahman} below for $r=1$).

In \cite{vDS1,vDS2}, van Diejen and the present author proposed
two multiple generalizations
of the integration formula \eqref{ebeta} in relation with
the root system $C_n$. The type I integral has the following form.
Let $z_1,\ldots,z_n\in\T$ and complex parameters
$t_1,\ldots,t_{2n+4}$ and $p,q$ satisfy the constraints
$|p|, |q|, |t_a|<1$ and $\prod_{j=1}^{2n+4}t_a=pq$. Then
\begin{eqnarray}\nonumber
&& \kappa_n\int_{\T^n}\prod_{1\leq j<k\leq n}\frac{1}{\eg(z_j^{\pm 1} z_k^{\pm 1};p,q)}
\prod_{j=1}^n\frac{\prod_{a=1}^{2n+4}\eg(t_az_j^{\pm 1};p,q)}
{\eg(z_j^{\pm2};p,q)}\prod_{j=1}^n\frac{dz_j}{z_j}
\\ && \makebox[4em]{}
=\prod_{1\leq a<b\leq 2n+4}\eg(t_at_b;p,q), \qquad
\kappa_{n}=\frac{(p;p)_\infty^n(q;q)_\infty^n}{(4\pi \textup{i})^n n!}.
\label{C-typeI}\end{eqnarray}
This integral evaluation can be considered as a high level generalization of
a Dixon identity \cite{Dixon}.

The type II $C_n$-integral has a structurally different form.
Let complex parameters $t, t_a (a=1,\ldots , 6), p$, and $q$ satisfy the
conditions $|p|, |q|,$ $|t|,$ $|t_a| <1,$ and
$t^{2n-2}\prod_{a=1}^6t_a=pq$. Then
\begin{eqnarray}\nonumber &&
\kappa_n\int_{\T^n} \prod_{1\leq j<k\leq n}
\frac{\eg(tz_j^{\pm 1} z_k^{\pm 1};p,q)}{\eg(z_j^{\pm 1} z_k^{\pm 1};p,q)}
\prod_{j=1}^n\frac{\prod_{a=1}^6\eg(t_az_j^{\pm 1};p,q)}{\eg(z_j^{\pm2};p,q)}
\prod_{j=1}^n\frac{dz_j}{z_j}
\\ && \makebox[4em]{}
= \prod_{j=1}^n\left(\frac{\eg(t^j;p,q)}{\eg(t;p,q)}
\prod_{1\leq a<b\leq 6}\eg(t^{j-1}t_at_b;p,q )\right).
\label{C-typeII}\end{eqnarray}
The latter integral can be interpreted as an elliptic extension
of the Selberg integral \cite{aar,for-war,Selberg}.
For $n=1$ both these multiple integrals reduce to the elliptic beta integral
\eqref{ebeta}. In the simplest $p\to 0$ limit, similar to the mentioned
reduction to the Rahman integral, one reproduces the Gustafson
$C_n$-integrals from \cite{gust2}.

\section{The rarefied elliptic gamma function}

An analogue of the elliptic gamma function
for the simplest lens space was introduced in \cite{ben}. Some functions
involving it were considered in \cite{kels,RW}.
In comparison to the standard elliptic hypergeometric integrals,
they contain some integer parameters and involve
finite summations over discrete variables additional to the standard
integrations. In this work we use the analysis of \cite{ben,kels,RW}
as an inspiration for considering the general structure of elliptic
hypergeometric functions of such type,
which we call the {\em rarefied} elliptic hypergeometric functions.

The lens space elliptic gamma function is determined by a particular product
of two standard elliptic gamma functions with different bases
\begin{eqnarray}\label{reg} &&
\gamma^{(r)}(z,m;p,q):=\Gamma(zp^m;p^r, pq)\Gamma(zq^{r-m};q^r, pq)
\\ && \makebox[4em]{}
=\prod_{j,k=0}^\infty\frac{1-z^{-1}p^{-m}(pq)^{j+1}p^{r(k+1)}}
{1-zp^m(pq)^{j}p^{rk}}
\frac{1-z^{-1}q^m(pq)^{j+1}q^{rk}}{1-zq^{r-m}(pq)^{j}q^{rk}},
\nonumber\end{eqnarray}
which, in addition to the variable $z\in\mathbb{C}^*$,
involves two integers $r\in\mathbb{Z}_{>0}$ and $m\in\mathbb{Z}$.
It has poles at the points
\beq
z_{\rm poles}=p^{-m-j-rk}q^{-j},\; p^{-j}q^{m-r-j-rk},
\label{zp}\eeq
and zeros
\beq
z_{\rm zeros}=p^{j+1+r(k+1)-m}q^{j+1},\; p^{j+1}q^{m+j+1+rk},\quad j,k\in\mathbb{Z}_{\geq 0}.
\label{zz}\eeq
According to \cite{ben},  this function is
associated with the superconformal index of a chiral superfield
on the space-time $S^1 \times L(r,k)$ with $k=-1$, where
$L(r,k)$ is the lens space defined by the identification of
points $(e^{2\pi\textup{i}/r} z_1,e^{2\pi\textup{i}k/r}z_2)\sim (z_1,z_2)$
in the complex representation of the $S^3$-sphere, $|z_1^2|+|z_2|^2=1$.
Erroneously, in the physics literature the same space is denoted
as $L(r,-k)$, which should not be confused with our notation.
To the present time, the superconformal
indices have only been computed for the space $L(r,-1)$. It is
worth of mentioning that the manifolds $L(r,-1)$
(or $L(r,r-1)$) and $L(r,1)$ are homeomorphic and differ by
orientation only. At the moment it is not clear whether the
rarefied elliptic hypergeometric functions
(superconformal indices) described below can distinguish them or not.

The function $\gamma^{(r)}(z,m;p,q)$
looks rather different from $\Gamma(z; p,q)$,
since it involves three bases $p^r, q^r, pq$ and the discrete variable $m$.
Let us show that, in fact, it is nothing but a special product of
standard  elliptic gamma functions with bases $p^r$ and $q^r$.
Consider the double elliptic gamma function $\Gamma(z;p,q,t)$
with a special choice of the third base parameter $t=pq$.
With its help, we can write
\baq \nonumber &&
\gamma^{(r)}(z,m;p,q)
=\frac{\Gamma(q^rz p^m;p^r,q^r,pq)}{\Gamma(z p^m;p^r,q^r,pq)}
\frac{\Gamma(p^r z q^{r-m};p^r,q^r,pq)}{\Gamma(z q^{r-m};p^r,q^r,pq)}
\\  && \makebox[6em]{}
=\frac{\Gamma((pq)^mq^{r-m}z;p^r,q^r,pq)}{\Gamma(q^{r-m}z;p^r,q^r,pq)}
\frac{\Gamma((pq)^{r-m}p^mz;p^r,q^r,pq)}{\Gamma(p^mz;p^r,q^r,pq)}.
\label{regm_egm3}\eaq
Using equation \eqref{2ndordereqn}, from the latter relation we
derive the product of two Pochhammer-type symbols built out of the elliptic
gamma function with the bases $p^r$ and $q^r$.
For $0\leq m\leq r$ we obtain the expression
\baq  &&
\gamma^{(r)}(z,m;p,q)
=\prod_{k=0}^{m-1}\Gamma(q^{r-m}z(pq)^k;p^r,q^r)\prod_{k=0}^{r-m-1}\Gamma(p^mz(pq)^k;p^r,q^r),
\label{regm_r1}\eaq
for $m<0$ we have
\baq &&
\gamma^{(r)}(z,m;p,q)=\frac{\prod_{k=0}^{r-m-1}\Gamma(p^mz(pq)^k;p^r,q^r)}
{\prod_{k=1}^{-m}\Gamma(q^{r-m}z(pq)^{-k};p^r,q^r)},
\nonumber\eaq
and for $m > r$,
\baq  &&
\gamma^{(r)}(z,m;p,q)=\frac{\prod_{k=0}^{m-1}\Gamma(q^{r-m}z(pq)^k;p^r,q^r)}
{\prod_{k=1}^{m-r}\Gamma(p^mz(pq)^{-k};p^r,q^r)}.
\nonumber\eaq
The second order elliptic gamma function is related to superconformal indices
of six dimensional field theories. Therefore it is natural to expect that
there exists some physical meaning of the function \eqref{regm_egm3}
from the point of view of compactification of
six dimensional theories to the lens space \cite{ben,RW}.

From \eqref{regm_r1} it follows that
for $r=1, m=0$ and $r=1, m=1$ we have the standard elliptic gamma function
$$
\gamma^{(1)}(z,0;p,q)=\gamma^{(1)}(z,1;p,q)=\Gamma(z;p,q).
$$
These equalities are related to the following factorized representation
of the elliptic gamma function
\begin{eqnarray}
\Gamma(z;p,q)=
\prod_{j=0}^\infty\frac{1-z^{-1}(pq)^{j+1}}{1-z(pq)^j}
\prod_{j,k=0}^\infty\frac{1-z^{-1}(pq)^{j+1}p^{k+1}}
{1-z(pq)^jp^{k+1}}
\frac{1-z^{-1}(pq)^{j+1}q^{k+1}}{1-z(pq)^{j}q^{k+1}}.
\label{r=1}\end{eqnarray}

For $r=1$ and $m\neq 0$ one can deduce directly from the definition \eqref{reg}
the recurrence relation
\begin{eqnarray}
\gamma^{(1)}(z,m+1;p,q)=\frac{\theta(zp^m;pq)}{\theta(zq^{-m};pq)}  \gamma^{(1)}(z,m;p,q),
\end{eqnarray}
yielding
\begin{eqnarray} \nonumber &&
\gamma^{(1)}(z,m;p,q)=\theta(z;pq|p)_m\theta(qz;pq|q)_{-m}\Gamma(z;p,q)
\\   && \makebox[6em]{}
=\left(-\frac{\sqrt{pq}}{z}\right)^{\frac{m(m-1)}{2}}
\left(\frac{q}{p}\right)^{\frac{m(m-1)(2m-1)}{12}}\Gamma(z;p,q), \quad m\in\mathbb{Z}.
\label{r=1gamma}\end{eqnarray}
As a result, the normalization condition $\Gamma(\sqrt{pq};p,q)=1$
is replaced by a more complicated relation
$$
\gamma^{(1)}(\sqrt{pq},m;p,q)=(-1)^{\frac{m(m-1)}{2}}
\left(\frac{q}{p}\right)^{\frac{m(m-1)(2m-1)}{12}}.
$$

Instead of the exact $(p,q)$-permutation symmetry one now has
\beq
\gamma^{(r)}(z,m;p,q)=\gamma^{(r)}(z,r-m;q,p).
\eeq
The $\gamma^{(r)}$-function has an important quasiperiodicity property
\begin{eqnarray} \nonumber &&
\frac{\gamma^{(r)}(z,m+kr;p,q)}{\gamma^{(r)}(z,m;p,q)}
=\prod_{l=0}^{k-1}\frac{\theta(zp^{m+lr};pq)}{\theta(zq^{-m-lr};pq)}
\\  && \makebox[4em]{}
=\left(-\frac{\sqrt{pq}}{z}\right)^{mk+r\frac{k(k-1)}{2}}
\left(\frac{q}{p}\right)^{k(\frac{1}{2}m^2+mr\frac{k-1}{2}
+r^2\frac{(k-1)(2k-1)}{12}) }, \quad k\in \mathbb{Z}.
\label{r-per_gen}\eaq
Since any integer $m$ can be represented in the form $l+kr$ with $0\leq l\leq r-1$, $k\in\mathbb{Z}$,
formulae \eqref{regm_r1} and \eqref{r-per_gen} provide general representation of the
$\gamma^{(r)}(z,m;p,q)$-function as a product of elliptic gamma functions with the bases $p^r$ and $q^r$
up to some (cumbersome, but elementary) exponential factor.

The inversion relation has the form
\beq
\gamma^{(r)}(z,m;p,q)\gamma^{(r)}(\textstyle{\frac{pq}{z}},r-m;p,q)=1,
\label{inv0}\eeq
which is proved by using the definition \eqref{reg}
and formula \eqref{inv1}.
The most elementary equations for this function have the form
\baq\nonumber &&
\gamma^{(r)}(qz,m+1;p,q)=\theta(zp^{m};p^r)\gamma^{(r)}(z,m;p,q),
\\ &&
\gamma^{(r)}(pz,m-1;p,q)=\theta(zq^{r-m};q^r)\gamma^{(r)}(z,m;p,q).
\label{rr}\eaq

Let us normalize the $\gamma^{(r)}(z,m;p,q)$-function as follows
\beq
\Gamma^{(r)}(z,m;p,q)
:=\left(-\frac{z}{\sqrt{pq}}\right)^{\frac{m(m-1)}{2}}
\left(\frac{p}{q}\right)^{\frac{m(m-1)(2m-1)}{12}}\gamma^{(r)}(z,m;p,q).
\label{regm}\eeq
We call this function {\em the rarefied elliptic gamma function}.
Its poles and zeros lie at the same points as in \eqref{zp} and \eqref{zz}.
For $r=1$, independently of the value of $m\in\mathbb{Z}$,  one has the equality
$$
\Gamma^{(1)}(z,m;p,q)=\Gamma(z;p,q),
$$
which is the source for the normalizing multiplier choice in \eqref{regm}.

The discrete variable quasiperiodicity takes now the form
\beq
\frac{\Gamma^{(r)}(z,m+kr;p,q)}{\Gamma^{(r)}(z,m;p,q)}
=\left[\left(-\frac{z}{\sqrt{pq}}\right)^{2m+rk}
 \left(\frac{p}{q}\right)^{m(m+rk)+\frac{r(2rk^2-1)}{6}}\right]^{\frac{1}{2}k(r-1)}, \quad k\in\mathbb{Z}.
\label{quasi1}\eeq
For this function the permutation of $p$ and $q$ is equivalent
to the change of sign $m\to -m$,
\beq
\Gamma^{(r)}(z,m;p,q)=\Gamma^{(r)}(z,-m;q,p).
\label{sympq}\eeq
The inversion relation also takes a natural compact form
\beq\label{invGamma}
\Gamma^{(r)}(z,m;p,q)\Gamma^{(r)}(\textstyle{\frac{pq}{z}},-m;p,q)=1.
\eeq
As a consequence one has the following relations
$$
\Gamma^{(r)}(\sqrt{pq},m;p,q)\Gamma^{(r)}(\sqrt{pq},-m;p,q)=1
$$
and
\beq
\Gamma^{(r)}(\sqrt{pq},0;p,q)=1,
\label{norm}\eeq
which can be used as a normalization condition.
For computing the residues we shall need the following limiting relation
 \beq\label{rfres} \qquad
\stackreb{\lim}{z\to 1}(1-z) \Gamma^{(r)}(z,0;p,q)
=\stackreb{\lim}{z\to 1}(1-z) \gamma^{(r)}(z,0;p,q)
=\frac{1}{(p^r;p^r)_\infty(q^r;q^r)_\infty},
 \eeq
which is easily established from the representation \eqref{regm_r1}, relation
\eqref{resid}, and the identity $\prod_{k=1}^{r-1} \Gamma((pq)^k;p^r,q^r)=1$.

The elementary recurrence relations take the form
\baq\nonumber && 
\Gamma^{(r)}(qz,m+1;p,q)=(-z)^m p^{\frac{m(m-1)}{2}}\theta(zp^{m};p^r)\Gamma^{(r)}(z,m;p,q),
\\ &&
\Gamma^{(r)}(pz,m-1;p,q)=(-z)^{-m}q^{\frac{m(m+1)}{2}}
\theta(zq^{-m};q^r)\Gamma^{(r)}(z,m;p,q).
\label{Gp}\eaq
The second equality is equivalent to the first one, but in the
explicit computations both forms are equally heavily used,
as well as the following relations
\baq\nonumber &&
\Gamma^{(r)}(q^{-1}z,m-1;p,q)=\left(-\frac{z}{pq}\right)^{-m}
p^{-\frac{m(m+1)}{2}}\frac{\Gamma^{(r)}(z,m;p,q)}{\theta(z^{-1}qp^{1-m};p^r)},
\\ &&
\Gamma^{(r)}(p^{-1}z,m+1;p,q)=\left(-\frac{z}{pq}\right)^m q^{-\frac{m(m-1)}{2}}
\frac{\Gamma^{(r)}(z,m;p,q)}{\theta(z^{-1}pq^{1+m};q^r)}.
\nonumber \eaq

Note that equations  \eqref{Gp}
do not determine uniquely the function $\Gamma^{(r)}(z,m;p,q)$.
The general solution of these equations has the form $\Gamma^{(r)}(z,m;p,q)\varphi_m(z)$,
where the functions $\varphi_m(z)$ satisfy the recurrences
$$
\varphi_{m+1}(qz)=\varphi_m(z), \qquad \varphi_{m-1}(pz)=\varphi_m(z).
$$
Resolution of the first equation yields $\varphi_m(z)=\varphi_0(q^{-m}z)$ for arbitrary function
$\varphi_0(z)$. The second equation yields $\varphi_{0}(pqz)=\varphi_0(z)$,
i.e. $\varphi_0(z)$
is an elliptic function of $z$ with the modular parameter $pq$,
$$
\varphi_0(z)=\prod_{k=1}^{K}\frac{\theta(\alpha_kz;pq)}{\theta(\beta_kz;pq)},
\qquad \prod_{k=1}^{K}\alpha_k= \prod_{k=1}^{K}\beta_k,
$$
for some integer $K=0,2,3,\ldots$ (the order of this elliptic function).
 Here the parameters $\alpha_k, \beta_k\in\mathbb{C}^*$
are arbitrary (up to one constraint) variables forming the divisor set of
$\varphi_0(z)$.
So, the space of solutions of interest has a functional freedom.
However, the quasiperiodicity condition \eqref{quasi1} removes
this freedom. Indeed, as a consequence of  \eqref{quasi1}
one gets the additional constraint $\varphi_0(q^rz)=\varphi_0(z)$.
For incommensurate $p$ and $q$, all indicated restrictions for
$\varphi_0(z)$ can be satisfied only by a constant, $\varphi_0(z)=const$,
which is fixed by the normalization condition \eqref{norm}.

For $0\leq m\leq r$, one can write
\baq \nonumber &&
\Gamma^{(r)}(z,m;p,q)=(-z)^{\frac{m(m-1)}{2}}p^{\frac{m(m-1)(m-2)}{6}} q^{-\frac{m(m^2-1)}{6}}
\\ && \makebox[5em]{} \times
\prod_{k=0}^{m-1}\Gamma(q^{r-m}z(pq)^k;p^r,q^r)\prod_{k=0}^{r-m-1}\Gamma(p^mz(pq)^k;p^r,q^r).
\label{GGrel}\eaq
This relation together with the quasiperiodicity \eqref{quasi1} expresses the
$\Gamma^{(r)}(z,m;p,q)$-function with arbitrary $m$ as a product
of ordinary elliptic gamma functions, since any $m$ can be represented
in the form $\ell+kr$ with $0\leq \ell \leq r-1$, $k\in\mathbb{Z}$.
Therefore it is natural to expect that all interesting integrals
constructed from the $\Gamma^{(r)}(z,m;p,q)$-function can be
related to some standard elliptic hypergeometric integrals.

Finally, we stress that the normalization of the $\gamma^{(r)}$-function
which we have chosen is not unique. It is easy to find
a multiplier $\mu(z,m)$ such that the product
$\tilde \Gamma^{(r)}(z,m;p,q):=\mu(z,m)\gamma^{(r)}(z,m;p,q)$ will be
$r$-periodic, $\tilde \Gamma^{(r)}(z,m+r;p,q)=\tilde \Gamma^{(r)}(z,m;p,q)$.
We have rejected this evident option from the very beginning
because in this case $\mu(z,m)$ will not be a meromorphic function of $z$
and in general one cannot write contour integrals for
products of $\tilde \Gamma^{(r)}$-functions. However, following an early version of
the present work, in \cite{KY} the rarefied elliptic hypergeometric functions
are written using such a periodic gamma function.
As a result, corresponding discrete parameters (as well as the discrete balancing
condition) are automatically defined modulo $r$, which is not so in our case.
The kernels of the corresponding sum-integrals become meromorphic functions of
the integration variables $z_j$ only after imposing the balancing condition.
As a result, in this case both normalizations of the gamma function appear
to be equivalent.

\section{A rarefied elliptic beta integral}

We define the kernel of a prospective rarefied elliptic beta integral
\beq
\Delta_\epsilon^{(r)}(z,m;t_a,n_a|p,q):=
\frac{\prod_{a=1}^6\Gamma^{(r)}(t_az,n_a+ m+\epsilon;p,q)
\Gamma^{(r)}(t_az^{-1},n_a- m;p,q)}
{\Gamma^{(r)}(z^{\pm2},\pm (2m+\epsilon);p,q)},
\label{ker_ebeta}\eeq
where we adopted the compact notation
\beq
\Gamma^{(r)}(tz^{\pm1},n\pm m;p,q):=\Gamma^{(r)}(tz,n+m;p,q)
\Gamma^{(r)}(tz^{-1},n-m;p,q).
\eeq
If the new discrete variable $\epsilon$ is an even integer,
then the transformation
$$
n_a\to n_a- \epsilon/2,\qquad  m \to m-\epsilon/2
$$
removes $\epsilon$ completely. For odd $\epsilon$ using this
transformation one can reduce the value of $\epsilon$ to 1.
Therefore we assume that $\epsilon$ takes the values
$$
\epsilon= 0\; \text{or}\; 1.
$$
One has the permutation symmetry
\beq
\Delta^{(r)}_\epsilon(z,m;t_a,n_a|p,q)=\Delta_{-\epsilon}^{(r)}(z,-m;t_a,-n_a|q,p).
\eeq
From now on, in most places we drop for brevity the bases $p$ and $q$ and the
superscript $r$ from the notation for integrands and rarefied elliptic gamma functions.

It is not difficult to derive the following equations
\beq
\frac{\Delta_\epsilon(pz,m-1;t_a,n_a)}{\Delta_\epsilon(z,m;t_a,n_a)}=h_1(z,m), \qquad
\frac{\Delta_\epsilon(qz,m+1;t_a,n_a)}{\Delta_\epsilon(z,m;t_a,n_a)}=h_2(z,m),
\label{rec12}\eeq
where
\baq  &&
h_1(z,m)=\left(\frac{q^{2m+\epsilon+1}}{pz^2}
\right)^{\sum_{a=1}^6n_a+3\epsilon}pq
\prod_{a=1}^6\frac{\theta(t_azq^{-n_a-m-\epsilon};q^r)}
{\theta(t_a^{-1}pzq^{1+n_a-m};q^r)}
\frac{\theta((pqz)^2q^{-2m-\epsilon};q^r)}{\theta(z^2q^{-2m-\epsilon};q^r)}
\nonumber\eaq
and
\baq   &&
h_2(z,m)=\left(\frac{p^{2m+\epsilon+1}}{qz^2}
\right)^{\sum_{a=1}^6n_a+3\epsilon}pq
\prod_{a=1}^6\frac{\theta(t_azp^{n_a+m+\epsilon};p^r)}
{\theta(t_a^{-1}qzp^{1-n_a+m};p^r)}
\frac{\theta((pqz)^2p^{2m+\epsilon};p^r)}{\theta(z^2p^{2m+\epsilon};p^r)}.
\nonumber\eaq
One can check that
\beq
\frac{h_1(q^rz,m)}{h_1(z,m)}=q^{2(1-r)(\sum_{a=1}^6n_a+3\epsilon)}
\frac{(pq)^2}{\prod_{a=1}^6t_a^2},
\qquad
\frac{h_2(p^rz,m)}{h_2(z,m)}=p^{2(r-1)(\sum_{a=1}^6n_a+3\epsilon)}
\frac{(pq)^2}{\prod_{a=1}^6t_a^2}.
\nonumber\eeq
Therefore, imposing the balancing condition
\beq
\prod_{a=1}^6t_a=pq, \qquad \sum_{a=1}^6n_a+3\epsilon=0,
\label{balance}\eeq
we obtain
\beq
h_1(q^rz,m)=h_1(z,m), \qquad  h_2(p^rz,m)=h_2(z,m).
\eeq
Denoting $q=e^{2\pi \textup{i} \sigma}$, $p=e^{2\pi \textup{i} \tau}$,
we see that $h_1(e^{2\pi \textup{i} u},m)$ becomes an elliptic function of $u$ of
order 10 with the periods 1 and $r\sigma$, while $h_2(e^{2\pi \textup{i} u},m)$
becomes a similar function with the periods
1 and $r\tau$. Note that one could fix the balancing condition for the
continuous parameters as $\prod_{a=1}^6t_a=-pq$, which also leads to
elliptic functions, but, similar to the standard elliptic beta integral
case \cite{spi:umn}, the choice \eqref{balance} is a distinguished one.

One can verify that now $\Delta_\epsilon(t_a,n_a;z,m)$
becomes a periodic function of $m$:
\beq\label{period}
\Delta_\epsilon(z,m+r;t_a,n_a)=\Delta_\epsilon(z,m;t_a,n_a),
\eeq
since all quasiperiodicity factors emerging from the relation \eqref{quasi1}
cancel out.
Using this fact we repeat $r$ times the recurrence relations \eqref{rec12}
and obtain
\begin{eqnarray} \label{eqD1} &&
\Delta_\epsilon(p^rz,m;t_a,n_a)=\prod_{k=0}^{r-1} h_1(p^kz,m-k)\,
 \Delta_\epsilon(z,m;t_a,n_a),
\\ &&
\Delta_\epsilon(q^rz,m;t_a,n_a)=\prod_{k=0}^{r-1} h_2(q^kz,m+k)\,
 \Delta_\epsilon(z,m;t_a,n_a).
\label{eqD2}\end{eqnarray}
Therefore, the function $\Delta_\epsilon(z,m;t_a,n_a)$ is a solution of a
finite-difference equation of the first order with the coefficient
given by a particular elliptic function of order $10r$.

We remind the reader of the definition of elliptic hypergeometric integrals
in the multiplicative notation \cite{spi:theta2} -- these are
the contour integrals $\int_C\Delta(z)dz/z$ with
$\Delta(z)$ satisfying the first order $q$-difference equation
$\Delta(qz)=h(z)\Delta(z)$ with a $p$-periodic (i.e., elliptic) function $h(z)$,
$h(pz)=h(z)$. Therefore, if we consider
a contour integral of our $\Delta_\epsilon$-function, by definition we obtain
a standard elliptic hypergeometric integral with the bases $p$ and $q$
replaced by $p^r$ and $q^r$, respectively.
If we further sum over $m$ we get an elliptic hypergeometric ``sum-integral".
We call such objects rarefied elliptic hypergeometric functions. Indeed,
they are represented by sums of the standard elliptic hypergeometric integrals
whose parameters are fixed in a particular way using the powers
of $p^{1/r}$ and $q^{1/r}$ (in the notation
of formula \eqref{ebeta}), which justifies the term ``rarefied".
Using the representation \eqref{GGrel} one can rewrite the function
$\Delta_\epsilon(z,m;t_a,n_a)$ as a ratio of standard elliptic gamma
functions with bases $p^r$ and $q^r$, however, the resulting expressions
are cumbersome and we do not consider them here.

A rarefied analogue of the elliptic beta integral \eqref{ebeta}
has the following evaluation.

\begin{theorem}
Let $t_1, \dots ,t_6,p,q \in {\mathbb{C}}^*$ and $n_1,\ldots, n_6\in\mathbb{Z}$
are such that $|t_a|,|p|,|q| <1$ and the following balancing condition holds true
$$
\prod_{a=1}^6 t_a=pq, \qquad \sum_{a=1}^6 n_a+3\epsilon=0,\quad
\epsilon=0,1.
$$
Then
\begin{equation} \label{rfint}
\kappa^{(r)}\sum_{m=0}^{r-1}
\int_{\mathbb{T}}\Delta_\epsilon^{(r)}(z,m;t_a,n_a)
\frac{dz}{z}
= \prod_{1 \leq a < b \leq 6} \Gamma^{(r)}(t_a t_b,n_a+n_b+\epsilon;p,q),
\end{equation}
where $\kappa^{(r)}=(p^r;p^r)_\infty (q^r;q^r)_\infty/4\pi \textup{i}$ and $\mathbb T$ is the positively oriented unit circle.
\end{theorem}

For $\epsilon=0$ one gets the relation proven by Kels in \cite{kels}.
The constraints $|t_a|<1$ can be relaxed by replacing for each fixed $m$
the integration contour $\T$ by a contour $C_m$ separating geometric
progressions of the $\Delta_\epsilon$-function poles
converging to zero from their reciprocals diverging to infinity.
The conditions of existence
of such contours are complicated, they impose certain constraints on
the parameters and require thorough considerations.
In this case, evidently, one cannot permute the summation over $m$ and
integration over $z$ in \eqref{rfint}. Note also that, because of the
periodicity \eqref{period}, the sum over $m=0,1,\ldots, r-1$ can be
replaced by sums over any $r$ consecutive values of the integer $m$.
Similarly, from the evident relation
$\Delta_\epsilon(z,-m;t_a,n_a)=\Delta_\epsilon(z^{-1},m-\epsilon;t_a,n_a)$
and the $r$-periodicity in $m$ one has
\begin{eqnarray*} &&
c_{r-m}:=\int_{\mathbb{T}}\Delta_\epsilon(z,r-m;t_a,n_a)\frac{dz}{z}=
\int_{\mathbb{T}}\Delta_\epsilon(z^{-1},m-\epsilon;t_a,n_a)\frac{dz}{z}
\\ && \makebox[2,5em]{}
= \int_{\mathbb{T}}\Delta_\epsilon(z,m-\epsilon;t_a,n_a)\frac{dz}{z}
=c_{m-\epsilon}.
\end{eqnarray*}
Therefore the sum over $m$ in \eqref{rfint} can be reduced for $\epsilon=0$ to
\beq
\sum_{m=0}^{r-1}c_m =
\left\{
\begin{array}{cl}
c_0+c_{r/2}+2\sum_{m=1}^{r/2-1}c_m & \text{for even}\, r, \\
c_0+2\sum_{m=1}^{(r-1)/2}c_m  & \text{for odd}\, r,
\end{array}
\right.
\label{e=0red}\eeq
and for $\epsilon=1$ to
\beq
\sum_{m=0}^{r-1}c_m =
\left\{
\begin{array}{cl}
2\sum_{m=0}^{r/2-1}c_m & \text{for even}\, r, \\
c_{(r-1)/2}+2\sum_{m=0}^{(r-3)/2}c_m  & \text{for odd}\, r.
\end{array}
\right.
\label{e=1red}\eeq
We shall use such a representation in Sect. 6 for the consideration
of a particular $p\to 0$ limit in the equality \eqref{rfint}.

According to the discussion given above, on the left-hand
side of identity \eqref{rfint} we have a sum of $r$ ordinary
elliptic hypergeometric integrals.
The proof of relation \eqref{rfint} different in certain aspects
from the one used (for $\epsilon=0$) in \cite{kels} will be given in the next section
as a subcase of a substantially more general situation.

\section{A $C_n$ rarefied elliptic beta integral of type I}

We define the kernel of the type I multiple rarefied elliptic beta integral
for the root system $C_n$ as
\baq \nonumber &&
\Delta_{I,\epsilon}({z}_j,{m}_j;{t}_a,{n}_a):=\prod_{1\leq j<k\leq n}
\frac{1}{\Gamma(z_j^{\pm 1}z_k^{\pm 1},\pm (m_j+\epsilon/2)\pm (m_k+\epsilon/2)))}
\\ && \makebox[2em]{} \times
\prod_{j=1}^n\frac{\prod_{a=1}^{2n+4}
\Gamma(t_a z_j,n_a+ m_j+\epsilon)
\Gamma(t_a z_j^{- 1},n_a-m_j)}
{\Gamma(z_j^{\pm 2},\pm (2m_j+\epsilon))},
\label{ker_ebetaCI}\eaq
where $ t_a,z_j\in \mathbb{C}^*,\, n_a, m_j\in\mathbb{Z},\, \epsilon=0,1,$
and impose the balancing condition
\beq
\prod_{a=1}^{2n+4}t_a=pq,\qquad \sum_{a=1}^{2n+4}n_a+(n+2)\epsilon=0.
\label{balI}\eeq
Association with the root system $C_n$ comes from the fact that the
denominator of the ratio of the products of rarefied elliptic gamma
functions in \eqref{ker_ebetaCI} can be formally written as
$\prod_{\alpha \in R(C_n)}\Gamma(e^{u\alpha},(m+\epsilon/2)\alpha;p,q)$,
where $\alpha\in\{\pm e_i\pm e_j\, (i<j), \pm 2e_i\}_{i,j=1,\ldots,n}$
are the roots of the root system $R(C_n)$ with $e_i$ being the standard
euclidean basis vectors of $\R^n$ and $z_i:=e^{ue_i}$, $m_i:=me_i$
for some formal variables $u\in\mathbb{C}$ and $m\in \mathbb{Z}$.

\begin{theorem}
Let $4n+8$ continuous and discrete parameters $t_a\in \mathbb{C}^*$,
$n_a\in \mathbb{Z}$,  $a=1,\ldots, 2n+4$, the variable $\epsilon=0,1$,
and bases $p,q\in \mathbb{C}$ satisfy the restrictions
$|p|, |q|, |t_a|<1$ and the balancing condition \eqref{balI}. Denote
$$
\kappa_{n}^{(r)}=\frac{(p^r;p^r)_\infty^n(q^r;q^r)_\infty^n}{(4\pi \textup{i})^n n!}.
$$
Then
\baq  && \makebox[-2em]{}
\kappa_{n}^{(r)} \sum_{m_1,\ldots,m_n=0}^{r-1}
\int_{\mathbb{T}^n}\Delta_{I,\epsilon}(z_j,m_j;t_a,n_a)
\prod_{j=1}^n\frac{dz_j}{z_j}
= \prod_{1 \leq a < b \leq 2 n+4} \Gamma(t_a t_b,n_a+n_b+\epsilon;p,q),
\label{rfintCI}\eaq
where $\mathbb T$ is the unit circle of positive orientation.
\end{theorem}

\begin{proof}
For proving this identity we adapt to the present situation the proof of the standard
type I $C_n$-integral evaluation formula given in \cite{spi:short}. As we will see,
several key steps will be identical. First, let us remove parameters $t_{2n+4}$
and $n_{2n+4}$ using the balancing constraint. For that we denote
$$
A:=\prod_{a=1}^{2n+3}t_a=\frac{pq}{t_{2n+4}}, \qquad
N:= \sum_{a=1}^{2n+3}n_a+(n+2)\epsilon= -n_{2n+4},
$$
and apply the inversion relation for the rarefied elliptic gamma function
involving parameter $t_{2n+4}$.
Now we divide the left-hand side of equality \eqref{rfintCI} by
$\kappa_{n}^{(r)}$ times the expression on the
right-hand side and rewrite it as
\beq\makebox[-2em]{}
I_{\epsilon}(t_1,\ldots, t_{2n+3},n_1,\ldots,n_{2n+3}):= \sum_{{m}=0}^{r-1}\int_{\T}
\rho_{\epsilon}(z_j,m_j;t_a,n_a)
\prod_{j=1}^n\frac{dz_j}{z_j}
=\frac{1}{\kappa_{n}^{(r)}},
\label{intI}\eeq
where
\baq \nonumber &&
\rho_{\epsilon}(z_j,m_j;t_a,n_a):=
\prod_{1\leq j<k\leq n}\frac{1}
{\Gamma(z_j^{\pm 1}z_k^{\pm 1},\pm (m_j+\epsilon/2)\pm (m_k+\epsilon/2))}
\\ && \makebox[2em]{} \times
\prod_{j=1}^n \frac{\prod_{a=1}^{2n+3}
\Gamma(t_a z_j,n_a+ m_j+\epsilon)
\Gamma(t_a z_j^{-1},n_a-m_j)}
{\Gamma(z_j^{\pm 2},\pm (2m_j+\epsilon))
\Gamma(Az_j^{-1},N- m_j-\epsilon)\Gamma(Az_j,N+ m_j)}
\nonumber \\ && \makebox[4em]{} \times
\frac{\prod_{a=1}^{2n+3}\Gamma(At_a^{-1},N-n_a-\epsilon)}
{\prod_{1 \leq a < b \leq 2n+3} \Gamma(t_a t_b,n_a+n_b+\epsilon)}.
\label{rho}\eaq
This $\rho_{\epsilon}$-function is $r$-periodic in all discrete variables
$m_j$, $j=1,\ldots,n,$  and $n_a$, $a=1,\ldots, 2n+3$,
\baq &&
\rho_{\epsilon}(z_j,\ldots, m_k+r, \ldots)=\rho_{\epsilon}(z_j,\ldots, n_b+r, \ldots)
=\rho_{\epsilon}(z_j,m_j;t_a,n_a),
\label{periodicity}\eaq
which is a very important property following from a lengthy cancellation
of the complicated quasiperiodicity multipliers generated by
the rarefied elliptic gamma functions.

Let us investigate the divisor of \eqref{rho} considered as a function
of $z_\ell$. Due to the property
$$
\Gamma^{(r)}(z,m)\Gamma^{(r)}(z^{-1},-m)=\frac{(pq)^{\frac{m(m+1)}{2}}}
{\theta(zq^{-m};q^r)\theta(z^{-1}p^{-m};p^r)},
$$
it does not contain poles whose positions do not depend on $t_a$ (at $z_\ell=0$
one has an essential singularity). The $t_a$-independent zeros do not play any
role in the following considerations and we skip them.
As to the $t_a$-dependent poles and zeros, generically, the function
$\rho_{\epsilon}({z}_\ell,\ldots)$ has sequences of poles converging
to $z_\ell=0$ for any $\ell$ by the points of the sets
$$
P_{\rm in}^A= \{t_aq^kp^{n_a-m_\ell+k+rj}\},\quad
 P_{\rm in}^B= \{t_aq^{r-n_a+m_\ell+k+rj}p^k\}
$$
with $a=1,\ldots, 2n+4$ and $j,k \in\mathbb{Z}_{\geq 0}$, and going to infinity
along the sets
$$
P_{\rm out}^A= \{t_a^{-1}q^{-k}p^{-n_a-m_\ell-\epsilon-k-rj}\},\quad
P_{\rm out}^B= \{t_a^{-1}q^{n_a+m_\ell+\epsilon-k-r(j+1)}p^{-k}\},
$$
which are not identically $z\to 1/z$ reciprocal to $P_{\rm in}$.
Zeros of this function converge to $z_\ell=0$ for any $\ell$
by the point sets
$$
Z_{\rm in}^A= \{t_a^{-1}q^{k+1}p^{-n_a-m_\ell-\epsilon+k+1+r(j+1)}\},\quad
Z_{\rm in}^B= \{t_a^{-1}q^{n_a+m_\ell+\epsilon+k+1+rj}p^{k+1}\}
$$
with $a=1,\ldots, 2n+4$ and $j,k \in\mathbb{Z}_{\geq 0}$, and go to infinity
along the point sets
$$
Z_{\rm out}^A= \{t_aq^{-k-1}p^{n_a-m_\ell-k-1-r(j+1)}\},\quad
Z_{\rm out}^B= \{t_aq^{-n_a+m_\ell-k-1-rj}p^{-k-1}\},
$$
which are also not identically $z\to 1/z$ reciprocal to $Z_{\rm in}$.

As one can see, the structure of poles and zeros is rather complicated and it may
happen that in the sets indicated above positions of some poles and zeros
coincide and, actually, both are absent. First, we assume that the parameters $t_a$ and bases $p,q$
are multiplicatively incommensurate, i.e. $t_a^nt_b^mp^kq^l\neq 1$
for $n,m,k,l\in \mathbb{Z}$, which guarantees that all poles and zeros are simple.
Then, equating positions of poles
and zeros (with $j,k$ replaced by $j',k'$),
we find that nontrivial cancellations may exist only if $j=j'$ and either
\begin{equation}\label{1}
rj+k+k'=-n_a+m_\ell -1,
\end{equation}
if there are intersecting points in $P_{\rm in}^A$ and $Z_{\rm out}^B$, or
\begin{equation}\label{2}
rj+k+k'=n_a-m_\ell -r-1,
\end{equation}
if $P_{\rm in}^B$ and $Z_{\rm out}^A$ overlap, or
\begin{equation}\label{3}
rj+k+k'=-n_a-m_\ell-\epsilon -1,
\end{equation}
if $P_{\rm out}^A$ and $Z_{\rm in}^B$ overlap, or
\begin{equation}\label{4}
rj+k+k'=n_a+m_\ell +\epsilon-r-1,
\end{equation}
if $P_{\rm out}^B$ and of $Z_{\rm in}^A$ overlap.
Let us denote as $p^{\rm max}_{\rm in}$ the maximal possible
absolute value of the pole positions
in some indicated subset of $P_{\rm in}$ and as $p^{\rm min}_{\rm out}$ the minimal
possible absolute value of the pole positions in some indicated subset of $P_{\rm out}$.

Recall that $j,k,k'\geq 0$, but $n_a$ and $m_\ell $ can take arbitrary integer values.
The periodicity \eqref{periodicity} means that the poles of the $\rho_{\epsilon}$-function
form a periodic lattice in
$n_a$ and $m_\ell$ and the above equations for $j, k, k'$ always have solutions for
sufficiently large $|n_a|$ and $|m|$, in which case a part of the poles is cancelled
by zeros. Therefore, without loss of generality, we can
restrict the values of $m$ and $n_a$ to
$$
0\leq m\leq r-1, \qquad -r< n_a< r
$$
(this can be done simply by the shifts $n_a\to n_a\pm r$
and $n_{2n+4}\to n_{2n+4}\mp r,$ as soon as one gets $|\sum_{k=1}^a(n_k+\epsilon/2)|\geq r$
for $a=1,\ldots, 2n+3$.
As a result, we have $-2r+1<n_a-m_\ell<r$ for all $a$, which means that
equation \eqref{2} has no solutions and
$p^{\rm max}_{{\rm in}, B}=\max |qt_a|$. Here and until formula \eqref{eqnCI} below we
assume that $\max$ and $\min$ values are taken in the parameter
sets with the index $a=1,\ldots, 2n+4$.

Suppose now that $n_a-m_\ell\geq 0$. Then equation \eqref{1} has no solutions
and $p^{\rm max}_{{\rm in}, A}=\max |t_a|$, which
is reached only for $m_\ell=n_a$ with the corresponding value of $a$.
Let now $m_\ell\geq n_a+1$. Then equation \eqref{1} may have nontrivial solutions.
For $j=0$ one has $k,k'=0,1,\ldots,m_\ell-n_a-1$, so that
$p^{\rm max}_{{\rm in}, A, j=0}=\max |qt_a|$. For $j=1$ and $0< m_\ell-n_a< r+1$
there are no solutions and the maximal absolute value of the corresponding
pole positions is $\max |t_a|$. For $j=1$ and $m_\ell-n_a\geq r+1$
(which can be satisfied only for $r>2$) the solution is
$k,k'=0,1,\ldots,m_\ell-n_a-1-r$, and the top possible pole position
has the absolute value $\max |qt_a|$. So, $p^{\rm max}_{{\rm in}, A, j=1}=\max |t_a|$.
The poles in $P_{\rm in}^A$ with $j>1$ have $p^{\rm max}_{{\rm in}, A, j>1}=\max |p^{2}t_a|$.
So, for $|t_a|<1$ all the poles from $P_{\rm in}$ lie inside $\T$.

Similar situation takes place for the poles $P_{\rm out}$ and zeros $Z_{\rm in}$. Indeed,
for $n_a+m_\ell+\epsilon \geq 0$ there are no solutions of equation \eqref{3} with
$p^{\rm min}_{{\rm out}, A}=\min |t_a^{-1}|$, which is reached
for $m_\ell=-n_a-\epsilon$  with the corresponding value of $a$.
For $m_\ell < -n_a-\epsilon$  the solution of \eqref{3} is $j=0$
and $k,k'=0,\ldots, -n_a-m_\ell-\epsilon-1$ with $p^{\rm min}_{{\rm out}, A, j=0}= \min |t_a^{-1}q^{-1}|$.
The poles with $j>0$ have $p^{\rm min}_{{\rm out}, A, j>0}= \min |t_a^{-1}p^{-1}|$
(for $\epsilon=0$).

Consider now $P_{\rm out}^B$. For $n_a+m_\ell+\epsilon\leq r$ equation \eqref{4}
has no solutions and $p^{\rm min}_{{\rm out}, B}=\min |t_a^{-1}|$,
which is reached for $m_\ell=r-n_a-\epsilon$ with the corresponding value of $a$.
For $n_a+m_\ell+\epsilon > r$  (which can be satisfied only for $r>2$) the solutions
of \eqref{4}  have the form $j=0$ and $k,k'=0,\ldots, n_a+m_\ell+\epsilon -r-1$
with $p^{\rm min}_{{\rm out}, B,j=0}=\min |t_a^{-1}p^{-1}|$. Finally, for $j>0$
equation \eqref{4} has no solutions and
one has $p^{\rm min}_{{\rm out}, B,j>0}=\min |t_a^{-1}q^{-1}|$ (for $\epsilon =1$). So, all
the poles from  $P_{\rm out}$ lie outside $\T$ for $|t_a|<1$.

To conclude, if we impose the constraint $\max |t_a|<1$, then
all poles from $P_{\rm in}$ and $P_{\rm out}$ are pushed inside and outside
of $\T$, respectively. It is exactly this property
(which is established after a rather neat analysis of
the structure of $\rho_{\epsilon}$-function divisor points)
that determines the choice
of $\T$ as the integration contour in formula \eqref{rfintCI}.

Now we prove the following finite-difference equation:
\baq \label{eqnCI} &&
\rho_{\epsilon}(z_j,m_j;pt_1,t_2,\ldots,n_1-1,n_2,\ldots)-\rho_{\epsilon}(z_j,m_j;t_a,n_a)
\\ && \makebox[4em]{}
=\sum_{k=1}^n\left(g_{k, \epsilon}(\ldots,p^{-1}z_k,\ldots, m_k+1,\ldots;t_a,n_a)
-g_{k, \epsilon}(z_j,m_j;t_a,n_a)\right),
\nonumber\eaq
where
\baq \nonumber &&
g_{k, \epsilon}(z_j,m_j;t_a,n_a)=\rho_{\epsilon}(z_j,m_j;t_a,n_a)
\prod_{\substack{\ell=1 \\ \ell\ne k}}^n
\frac{\theta(t_1z_\ell q^{-n_1-m_\ell-\epsilon},
t_1z_\ell^{-1}q^{-n_1+m_\ell};q^r)}
{\theta(z_kz_\ell q^{-m_k- m_\ell-\epsilon},z_kz_\ell^{- 1}q^{-m_k+ m_\ell};q^r)}
\\ && \makebox[2em]{}  \times
\frac{\prod_{a=1}^{2n+3}\theta(t_az_kq^{-n_a-m_k-\epsilon};q^r)}
     {\prod_{a=2}^{2n+3}\theta(t_1t_aq^{-n_1-n_a-\epsilon};q^r)}
\frac{\theta(t_1Aq^{-n_1-N};q^r)}{\theta(z_k^2q^{-2m_k-\epsilon},Az_kq^{-N-m_k};q^r)}
\frac{t_1q^{m_k}}{z_kq^{n_1}}.
\label{gCI}\eaq

Dividing this equation by $\rho_{\epsilon}(z_j,m_j;t_a,n_a)$ we come to the following identity
\begin{eqnarray}\nonumber  &&
\prod_{j=1}^n\frac{\theta(t_1z_jq^{-n_1- m_j-\epsilon},t_1z_j^{- 1}q^{-n_1+ m_j};q^r)}
{\theta(Az_jq^{-N-m_j},Az_j^{-1}q^{-N+ m_j+\epsilon};q^r)}
\prod_{a=2}^{2n+3}\frac{\theta(At_a^{-1}q^{-N+n_a+\epsilon};q^r)}
                       {\theta(t_1t_aq^{-n_1-n_a-\epsilon};q^r)} - 1
\\ \nonumber && \makebox[2em]{}
=\frac{t_1q^{-n_1}\theta(t_1Aq^{-n_1-N};q^r)}
{\prod_{a=2}^{2n+3}\theta(t_1t_aq^{-n_1-n_a-\epsilon};q^r)}
\sum_{k=1}^n\frac{q^{m_k}}{z_k\theta(z_k^2q^{-2m_k-\epsilon};q^r)}
\\ \nonumber && \makebox[2em]{} \times
\prod_{\substack{j=1 \\ j\ne k}}^n
\frac{\theta(t_1z_jq^{-n_1- m_j-\epsilon},t_1z_j^{-1}q^{-n_1+ m_j};q^r)}
{\theta(z_kz_jq^{-m_k- m_j-\epsilon},z_kz_j^{-1}q^{-m_k+m_j};q^r)}
\\  && \makebox[2em]{}
\times
\left(\frac{z_k^{2n+2}\prod_{a=1}^{2n+3}\theta(t_az_k^{-1}q^{-n_a+m_k};q^r)}
{q^{(n+1)(2m_k+\epsilon)}\theta(Az_k^{-1}q^{-N+m_k+\epsilon};q^r)} -
\frac{\prod_{a=1}^{2n+3}\theta(t_az_kq^{-n_a-m_k-\epsilon};q^r)}
{\theta(Az_kq^{-N-m_k};q^r)}\right).
\label{eqn-CI}\end{eqnarray}
The shifts $z_j\to z_jq^{m_j+\epsilon/2}$ and $t_a\to t_a q^{n_a+\epsilon/2}$ remove completely
the discrete variables $m_j$, $n_a$ and $\epsilon$ from \eqref{eqn-CI} and we obtain precisely
the elliptic functions identity established in \cite{spi:short} in
the proof of the type I $C_n$-integral (with $p$ and $q$
replaced by $p^r$ and $q^r$).

We now integrate equation \eqref{eqnCI} over the multi-contour $\T^n$
and sum over all $m_j$ from $0$ to $r-1$. It can be checked that
$g_{k, \epsilon}$-functions
are periodic with respect to the shifts $m_j\to m_j+r$. Therefore we obtain
\baq \nonumber &&
I_{\epsilon}(pt_1,t_2,\ldots, t_{2n+3},n_1-1,n_2,\ldots,n_{2n+3})-I_{\epsilon}(t_a,n_a)
\\ && \makebox[2em]{}
=\sum_{m_1,\ldots,m_n=0}^{r-1} \sum_{\ell=1}^{n}
\left(\int_{\T^{\ell-1}\times(p^{-1}\T)\times \T^{n-\ell} }
-\int_{\T^n}\right)
g_{\ell, \epsilon}(z_j,m_j;t_a,n_a)\prod_{j=1}^n\frac{dz_j}{z_j},
\label{int-eqnCI}\eaq
where $p^{-1}\T$ denotes the contour obtained from $\T$ after
scaling it by $p^{-1}$.

The divisor points of the functions $g_\ell$ (\ref{gCI}) in the variable $z_\ell$
are determined by the following factor (theta functions were absorbed into
the gamma functions by appropriate shifts of the arguments):
\begin{eqnarray*} &&
\prod_{a=1}^{2n+3}\Big[\Gamma(pt_a z_\ell,n_a+m_\ell+\epsilon-1)
\Gamma(t_a z_\ell^{-1},n_a-m_\ell)\Big]
\\ && \makebox[1em]{} \times
\Gamma(t_{2n+4}z_\ell,n_{2n+4}+m_\ell+\epsilon)
\Gamma(p^{-1}t_{2n+4} z_\ell^{-1},n_{2n+4}-m_\ell+1).
\end{eqnarray*}
Comparing with the previous analysis of the divisor of the $\rho_{\epsilon}$-function,
we see that the equations \eqref{1} and \eqref{2} are preserved for
$P_{\rm in}$ poles associated with $t_a,\, a=1,\ldots,2n+3$, whereas, vice versa,
equations \eqref{3} and \eqref{4} remain the same for $P_{\rm out}$ poles
associated with $t_{2n+4}$. The sum $rj+k+k'$ is equal to $-n_{2n+4}+m_\ell-2$
or $n_{2n+4}-m_\ell-r$ for the analogues of \eqref{1} and \eqref{2}
with $a=2n+4$, respectively, and to $-n_{a}-m_\ell-\epsilon$
or $n_{a}+m_\ell+\epsilon-r-2$ for the analogues of \eqref{3} and \eqref{4}
with $a=1,\ldots, 2n+3$. As a result of such changes we find that
$p^{\rm max}_{\rm in}=\max\{ |t_a|,|p^{-1}t_{2n+4}|\}$ and
$p^{\rm min}_{\rm out}=\min\{ |t_{2n+4}^{-1}|,|p^{-1}t_{a}|\}$, where $a=1,\ldots, 2n+3$.

Therefore, for $|t_a|<1,\, a=1,\ldots, 2n+3$, and $|t_{2n+4}|<|p|$
the functions $g_\ell$ do not have poles in the annuli
$1\leq |z_\ell|\leq |p^{-1}|$. As a result, we can safely shrink
the integration contour $p^{-1}\T$ to $\T$ in \eqref{int-eqnCI}
and obtain zero on the right-hand side, i.e. the equality
\beq
I_{\epsilon}(pt_1,t_1,\ldots,n_1-1,n_2,\ldots)=I_{\epsilon}(t_a,n_a).
\label{rec1}\eeq
Note that  for the taken constraints on the parameters the contour $\T$ is
legitimate for both integrals on the left-hand side of \eqref{int-eqnCI},
i.e. it separates relevant sets of poles.

Due to the incommensurability condition, the integral $I_{\epsilon}(t_a,n_a)$
is a meromorphic function of the parameters $t_a$. Therefore,
equation \eqref{rec1} can be used for analytic continuation
of $I_{\epsilon}(t_a,n_a)$ from the domain $|t_1|, |t_{2n+4}|<1$ to
$|p^kt_1|, |p^{-k}t_{2n+4}|<1$ for any $k\in\mathbb{Z}$.
Therefore, iterating \eqref{rec1} $r$ times and using
the periodicity property
$$
I_{\epsilon}(\ldots, n_{b-1},n_b+r,n_{b+1},\ldots)=I_{\epsilon}(t_a,n_a), \quad b=1,\ldots, 2n+3,
$$
following from the $\rho_{\epsilon}$-function periodicity in variables $n_b$,
we obtain the equality
\beq
I_{\epsilon}(p^rt_1,t_2,\ldots,n_a)=I_{\epsilon}(t_a,n_a).
\label{p-shift}\eeq

Let us impose the additional constraint $|t_{2n+4}|<|q|$.
Then we can permute bases $p$ and $q$ in the above considerations,
apply the symmetry \eqref{sympq} and obtain the equality
\beq
I_{\epsilon}(q^rt_1,t_2,\ldots,-\epsilon-n_a)=I_{\epsilon}(t_a,-\epsilon-n_a).
\label{q-shift}\eeq
However, the integral $I_{\epsilon}$ is $r$-periodic in $n_a$-variables, i.e. the signs of $n_a$
 do not matter. Therefore, we have
$$
I_{\epsilon}(p^rt_1,t_2,\ldots,n_a)=I_{\epsilon}(q^rt_1,t_2,\ldots,n_a)=I_{\epsilon}(t_1,t_2,\ldots,n_a).
$$
Since our bases $p$ and $q$ are incommensurate, this means that $I_{\epsilon}(t_a,n_a)$
does not depend on the parameter $t_1$ and, by symmetry, on all parameters $t_a$.
Substituting this condition into recursion \eqref{rec1}, we find that, actually,
$I_{\epsilon}(t_a,n_a)$ does not depend on $n_a$ as well, i.e. it is a constant depending
only on $p, q, r, n $ and $\epsilon$, $I_{\epsilon}(t_a,n_a)= c_\epsilon(p, q, r,n)$.
Let us compute this constant $c_\epsilon$.

For that we set
$$
n_a=0,\, a=1,\ldots, n+2, \qquad n_a=-\epsilon,\, a=n+3,\ldots, 2n+4,
$$
which satisfies the discrete balancing condition,  and consider the limit
$$
t_at_{a+n+2}\to 1,\quad a=1,\ldots,n.
$$
Our analysis of the $\rho_{\epsilon}$-function
divisor structure shows that in each summation over the
discrete variables $0\leq m_j\leq r-1$ there is one integral, corresponding
to the value $m_j=0$, for which the integration contour $\T$ becomes pinched
by $2n$ pairs of poles. The $\rho_{\epsilon}$-function contains the factor
$1/\prod_{j=1}^n\Gamma(t_jt_{j+n+2},0)$ which vanishes unless it is
cancelled by the residues of poles pinching the integration contour
for all $n$ integrals simultaneously. Therefore, our problem reduces
to computation of the limit
\begin{eqnarray*} &&
\stackreb{\lim}{t_at_{a+n+2}\to 1\atop a=1,\ldots,n}\mu_\epsilon(t_a)
\int_{\T^n}\prod_{1\leq j<k\leq n}\frac{1}{\Gamma((z_jz_k)^{\pm 1},\pm\epsilon)
\Gamma((z_j/z_k)^{\pm 1},0)}
\\ && \makebox[2em]{} \times
\prod_{j=1}^n \frac{\prod_{a=1}^{n+2}\Gamma(t_a z_j,\epsilon)\Gamma(t_{a+n+2} z_j,0)
\Gamma(t_a z_j^{- 1},0)\Gamma(t_{a+n+2} z_j^{- 1},-\epsilon)}
{\Gamma(z_j^{\pm 2},\pm\epsilon)}\frac{dz_j}{z_j},
\end{eqnarray*}
where
$$
\frac{1}{\mu_\epsilon(t_a)}:=\prod_{1 \leq a < b \leq n+2} \Gamma(t_a t_b,\epsilon)
 \Gamma(t_{a+n+2} t_{b+n+2},-\epsilon)\prod_{1 \leq a, b \leq n+2} \Gamma(t_a t_{b+n+2},0).
$$
Before taking the limits $t_at_{a+n+2}\to 1,\,a=1,\ldots, n$,
we deform each $\T$ to a contour $C$ which crosses the poles
$z_j=t_a, t_{a+n+2},\, a=1,\ldots, n$, and does not touch other poles.
Again, the result does not vanish for $t_at_{a+n+2}=1$ only
if we pick up residues for all variables $z_j$ simultaneously.
Whenever two different variables $z_j$ pick up
residues from identical pole positions, we get zero
due to the functions $\Gamma((z_j/z_k)^{\pm 1},0)$ in
the integrand's denominator. Therefore, we should consider only the
residues for $z_j=t_j$ and their $n$! permutations giving identical results.
The residues for $z_j=t_{j+n+2}$ give the same result,
which results in the additional multiplier $2^n$.
So, the limit of interest is equal to
\begin{eqnarray*} &&  \makebox[-2em]{}
n!(4\pi \textup{i})^n\stackreb{\lim}{t_at_{a+n+2}\to 1\atop a=1,\ldots,n}
\mu_\epsilon(t_a)
\frac{\prod_{j=1}^n \prod_{a=1}^{n+2}
\Gamma(t_a t_j,\epsilon)\Gamma(t_{a+n+2} t_j,0)
\Gamma(t_{a+n+2} t_j^{- 1},-\epsilon)
}
{\prod_{1\leq j<k\leq n}\Gamma((t_jt_k)^{\pm 1},\pm \epsilon)
\Gamma((t_j/t_k)^{\pm 1},0) \prod_{j=1}^n\Gamma(t_j^{\pm 2},\pm\epsilon) }
\\ && \makebox[4em]{} \times
\prod_{j=1}^n \prod_{a=1}^{n+2}
\stackreb{\lim}{z_j\to t_j}(1-t_jz_j^{-1})\Gamma(t_a z_j^{-1},0).
\end{eqnarray*}
We apply now the limiting relation \eqref{rfres} and cancel common factors
from the numerator and denominator. The remaining gamma functions
disappear after substitution of the relations $t_at_{a+n+2}=1,\, a=1,\ldots, n+2,$
and $t_{n+1}t_{n+2}t_{2n+3}t_{2n+4}=pq$ due to the inversion relation
\eqref{invGamma}. As a result,
$$
c_\epsilon(p, q, r, n)=\stackreb{\lim}{t_at_{a+n+2}\to 1 \atop a=1,\ldots,n}
I_{\epsilon}(t_j,n_j)\Big|_{n_j=n_{j+n+2}+\epsilon=0 \atop j=1,\ldots,n+2}
=
\frac{(4\pi \textup{i})^n n!}{(p^r;p^r)_\infty^n(q^r;q^r)_\infty^n}
=\frac{1}{\kappa_{n}^{(r)}},
$$
as required.

Finally, by analytic continuation, we relax the restrictions
$|t_{2n+4}|<|p|,|q|$ to $|t_{2n+4}|<1$ and remove the incommensurability
constraint $t_a^nt_b^mp^kq^l\neq 1$, $n,m,k,l\in \mathbb{Z}$
(still keeping $|t_{a}|,|p|,|q|<1$). The theorem is proved.
\end{proof}

Evidently, in the final result \eqref{rfintCI} one can relax
restrictions for the $t_a$-parameters values by changing $\T$ to a
contour $C$ such that it separates the poles
$P_{\rm in}$ and  $P_{\rm out}$ for all possible values of $m_\ell$.
However, the analysis of sufficiency conditions for existence of such
a contour is a complicated task and we do not consider it here.

\section{A special $p\to 0$ limit}

Let us rewrite $n=1$ relation \eqref{rfint} in terms of the $\gamma^{(r)}$-functions.
The left-hand side expression takes the form
$$
\prod_{a=1}^6 \left[\frac{t_a^{(n_a+\epsilon/2)(n_a+\epsilon/2-1)}}{p^{(n_a+\epsilon/2)^2}}
\left(\frac{p}{q}\right)^{\frac{1}{3}(n_a+\epsilon/2)^3}\right]
\sum_{m=0}^{r-1}c_m(n_a,\epsilon),
$$
where
$$
c_m(n_a,\epsilon)=\kappa^{(r)}
\left(\frac{q}{p}\right)^{(m+\epsilon/2)^2}\int_{\mathbb{T}}
\frac{\prod_{a=1}^6\gamma^{(r)}(t_az,n_a+m+\epsilon)\gamma^{(r)}(t_az^{-1},n_a-m)}
{z^{2m+\epsilon}\gamma^{(r)}(z^{\pm2},\pm (2m+\epsilon))}\frac{dz}{z}
$$
(note that the integrand is an explicit meromorphic function of $z\in\mathbb{C}^*$).
The right-hand side can be ``simplified"
in a similar way and, after cancelling common factors, we obtain the relation
\baq   && 
\sum_{m=0}^{r-1}c_m(n_a,\epsilon)
=(-1)^\epsilon\prod_{a=1}^6 \frac{(q/p)^{\frac{1}{2}(n_a+\epsilon/2)^2}}{t_a^{n_a+\epsilon/2}}
\prod_{1\leq a<b\leq 6}\gamma^{(r)}(t_at_b,n_a+n_b+\epsilon).
\label{intfingamma}\eaq

Now we can compare our proof of this relation for $\epsilon=0$
with the one suggested by Kels in \cite{kels}.
In \cite{kels} the function $[\! [ m ]\!] := m \mod r$ is used
which is not needed in our consideration, since all
necessary finite-dimensional truncations
are guaranteed by the $r$-periodicity in all discrete variables of the
integrand in \eqref{intI}. This makes our formulae uniform in the values of discrete variables
$n_a$ and $m$ in contrast with \cite{kels}, where they
were brought to the values $1,\ldots, r-1$ by subtracting or adding multiples of $r$.

The kernel of sum-integral of \cite{kels} has the same $\gamma^{(r)}$-dependent part
as above (with $\epsilon=0$), but the multiplier in front of it looks substantially different.
In particular, for $n_a=0$ it misses our $z^{-2m}$ factor. However, discrete
arguments of the corresponding gamma functions involve $[\![ \ell]\!] $,
for some integers $\ell$, instead of simple $\ell$ in our case.
It appears, that after replacing $[\![ \ell ]\!] $ by $\ell+kr$ for appropriate
values of $k$ bringing $\ell+kr$ to
the values $0,1,\ldots,r-1$ and applying the quasiperiodicity
relations \eqref{r-per_gen}, our identity \eqref{rfint} coincides in this case
with the one suggested in \cite{kels}.

Since we use different normalizing factor for the rarefied elliptic gamma
function \eqref{regm}, even after clearing
the arguments of gamma functions from $[\![ \ell]\!] $-functions,
the left-hand and right-hand sides of our identities do not coincide
(but the ratio does). Next, an analogue of the relation \eqref{eqn-CI}
in \cite{kels} contains formal fractional powers of $z$ and it looks different
from the one in \cite{spi:short}, which we use. However, again, after
replacing $[\![ \ell ]\!] $ by appropriate expressions $\ell+kr$, these fractional
powers disappear and the finite-difference equation used in  \cite{kels}
becomes identical with ours. So, the approach suggested in this work
is equivalent to the one in \cite{kels}, but, to the author's taste,
it looks more natural. Another point is that the analysis of the divisor
points of the integrands in our case is essentially more detailed,
since we check case by case the possibilities of cancellation of
poles and zeros, whereas is \cite{kels} the divisor was forced to take a
particular form by the use of $[\![ \ell ]\!] $-functions.

Inspired by the first version of this paper, the authors of \cite{GahKels}
suggested to use the following periodic gamma function
\baq &&
\tilde \Gamma^{(r)}(z,m;p,q):= \left(\frac{z}{\sqrt{pq}}\right)^{\frac{m(m-r)}{2r}}
\left(e^{-\pi \textup{i}}\sqrt{\frac{p}{q}}\right)^{\frac{m(m-r)(2m-r)}{6r}}\gamma^{(r)}(z,m;p,q),
\label{pergamma}
\\ && \makebox[6em]{}
\tilde \Gamma^{(r)}(z,m+r;p,q)=\tilde \Gamma^{(r)}(z,m;p,q),
\label{periodicrel} \eaq
for rewriting the sum-integrals without employing the $[\![m]\!]$-function.
Note that our normalization factor coincides with the above one for $r=1$.  Clearly,
this is a multivalued function of its arguments and a particular
branch of the power of $-1=e^{-\pi \textup{i}}$ was chosen in the
considerations of \cite{GahKels}. Therefore the sum-integrals obtained
after replacing our $\Gamma^{(r)}(z,m;p,q)$ by \eqref{pergamma} will be multivalued as well.
Still, the balancing condition removes this drawback for the integration variable
$z$ and, as claimed in \cite{GahKels}, the $\epsilon=0$ rarefied elliptic beta
integral evaluation formula written in terms of \eqref{pergamma} yields
exactly the same identity as derived in the present paper. The author
checked that this is true for $\epsilon\neq 0$ as well.

Due to the periodicity, for sum-integrals written in terms of function \eqref{pergamma}
the discrete balancing condition takes the form $\sum_{a=1}^6 (n_a+\epsilon/2)=Kr$,
for arbitrary $K\in\mathbb{Z}$. Using the quasiperiodicity of
the $\gamma^{(r)}(z,m;p,q)$-function it is possible to remove this integer $K$
by a simple shift $n_a\to n_a+Kr$ for any fixed $a$. Therefore
our balancing condition (corresponding to $K=0$) is a generic one.
As a result, the approaches developed in this paper and the one of \cite{GahKels,KY} are equivalent.
An advantage of the present consideration is that all our sum-integrals are single-valued
functions of parameters, which is not so in the other case.

Let us stress that the choice of the normalization factor leading to the
periodicity \eqref{periodicrel}
is not unique. If we flip the minus sign from $\sqrt{p/q}$ to $\sqrt{pq}$,
i.e. replace in \eqref{pergamma}  $e^{-\pi \textup{i} \frac{m(m-r)(2m-r)}{6r} }$
by $e^{-\pi \textup{i} \frac{m(m-r)}{2r} }$, then the coincidence with our results
breaks down, since this will bring a nontrivial $m$-dependent multivalued factor inside
the summation $\sum_{m=0}^{r-1}$.  So, not all periodic gamma
functions work equally well. It is necessary to understand the reasons
distinguishing the choice \eqref{pergamma}.

Assume now that $r>1$ and consider the simplest $p\to 0$ limit in the above
relation \eqref{intfingamma} for a special choice of discrete parameters
$$
n_{1,2,3}=0, \, \quad
n_{4,5,6}=-\epsilon.
$$
For that we substitute into the arguments of $\gamma^{(r)}$-functions the
relation $t_6=pq/\prod_{a=1}^5t_a$ and use the inversion relation \eqref{inv0}.
Now we can take the limit $p\to 0$ for fixed $t_1,\ldots,t_5$.
Whenever the modulus of discrete arguments of the $\gamma^{(r)}$-functions
is bounded by $r$, we use the asymptotic relation
\beq \label{asymp}
\gamma^{(r)}(z,m;p,q) \stackreb{\to}{p\to 0}
\left\{
\begin{array}{cl}
(zq^{r-m};q^r)_\infty^{-1}, &   0< m \leq r,  \\
(z;q^r)_\infty^{-1}, &   m=0,  \\
\left(\frac{-p}{z}\right)^{|m|}
\frac{(p/q)^{|m|(|m|-1)/2}}{(zq^{|m|};q^r)_\infty}, &   -r< m<0.
\end{array}
\right.
\eeq
For other values of $m$ we apply first the quasiperiodicity relation \eqref{r-per_gen}
before taking the asymptotics.
Because of the symmetry $c_{r-m}=c_{m-\epsilon}$, it is sufficient to consider the asymptotics of
the coefficients $c_0,\ldots,c_\ell$ with $\ell=r/2$ (for even $r$) or $\ell=(r-1)/2$ (for odd $r$),
provided $\epsilon=0$. For $\epsilon =1$ one has either $\ell=r/2-1$ (for even $r$) or $\ell=(r-1)/2$
(for odd $r$), see \eqref{e=0red} and \eqref{e=1red}.
Combining all asymptotic terms we obtain the following picture.

For $\epsilon=0$ and $m=0,\ldots, [r/2]$ (an integral part of $r/2$) the coefficients
$c_m\propto p^{m-1/4},\, p\to0$ for both even and odd $r$.
As a result, in the limit of interest the leading asymptotics of $c_m$ corresponds to $m=0$.
After considering the leading asymptotics of the right-hand side expression in \eqref{intfingamma},
we come to the following identity (recall that $|t_a|<1$)
\beq
\frac{(q^r;q^r)_\infty}{4\pi \textup{i} }\int_{\mathbb{T}}
\frac{(\prod_{b=1}^5t_b z^{\pm1};q^r)_\infty(z^{\pm2};q^r)_\infty}
{\prod_{a=1}^5 (t_az^{\pm1};q^r)_\infty}
\frac{dz}{z}
=\frac{\prod_{a=1}^5(t_a^{-1}\prod_{b=1}^5t_b;q^r)_\infty}
{\prod_{1\leq a<b\leq 5}(t_at_b;q^r)_\infty},
\label{rahman}\eeq
which is the Rahman $q$-beta integral evaluation formula \cite{Rahman} for the base $q^r$.
We conclude that in this case the rarefied elliptic beta integral does not produce new objects.

However, for $\epsilon=1$ the situation is qualitatively different.
First we note that in this case our choice of the discrete variables $n_a$
breaks $S_6$-symmetry. Therefore the consideration depends
on whether we take $t_6\to 0$ as $p\to0$ (as above), or if one of $t_a,\, a=1,2,3,$
tends to zero. Keeping $t_1,\ldots,t_5$ fixed we repeat the same steps as above.
Then for even $r$ we find the asymptotics $c_m\propto p^{m+3/4},\, p\to0$ with $m=0,\ldots, r/2-1.$
For odd $r$ we find  $c_m\propto p^{m+3/4},\, p\to0$ with $m=0,\ldots, (r-1)/2.$
In both cases the leading term corresponds to $m=0$.
As follows from \eqref{e=1red} such a term enters twice the full sum
over $m$. Computing asymptotics of the right-hand side
expression, we come to the same leading behaviour.
After cancelling diverging factors there emerges a nontrivial relation.
Omitting the technical details of computations for both sides of the identity \eqref{intfingamma},
we come to the following statement.
\begin{theorem}
Let $r\in\mathbb{Z}_{>1}$ and the variables $t_1,\ldots,t_5, q\in\mathbb{C}$ are such that $|t_a|, |q|<1$. Then
\begin{eqnarray} \label{NEW}&&
\frac{(q^r;q^r)_\infty}{2\pi \textup{i} }\int_{\mathbb{T}}
\frac{(q^{r-1}Az,Az^{-1},q^{r-1}z^2,q z^{-2};q^r)_\infty}
{\prod_{a=1}^3 (q^{r-1}t_az,t_az^{-1};q^r)_\infty\prod_{a=4}^5 (t_az,q t_a z^{-1};q^r)_\infty}
\frac{dz}{z}
\\ && \makebox[2em]{}
=\frac{\prod_{a=1}^3(At_a^{-1};q^r)_\infty \prod_{a=4}^5(q^{r-1}At_a^{-1};q^r)_\infty}
{\prod_{1\leq a<b\leq 3}(q^{r-1}t_at_b;q^r)_\infty
\prod_{a=1}^3\prod_{b=4}^5(t_at_b;q^r)_\infty (qt_4t_5;q^r)_\infty}, \quad A=\prod_{a=1}^5t_a.
\nonumber \end{eqnarray}
\end{theorem}
Choosing $n_{1,2,3}=-\epsilon$ and $n_{4,5,6}=0$ and repeating the same limiting procedure,
we find again relation \eqref{NEW} with flipped multipliers $q$ and $q^{r-1}$ in the
arguments of $q$-shifted factorials.
The derived formulae represent a new layer of $q$-hypergeometric  identities
which was not considered in the $q$-treatise \cite{GR}.
It is distinguished by the breaking of $S_5$ symmetry of \eqref{rahman}
and, especially, of the $z\to z^{-1}$ symmetry of the integrand, which
is quite unusual. Note that the restriction $r>1$ in \eqref{NEW} comes from
the fact that for $r=1$ the original relation \eqref{intfingamma}
reduces to the standard elliptic beta integral which does not yield new
relations in the $p\to 0$ limit.

Let us denote $t_a=q^{ru_a}$, $A=q^{rU}$, and
take the limit $q\to 1^-$ in \eqref{NEW} for fixed $u_a$. Using the uniform limiting relation
for the Jackson $q$-gamma function \cite{aar,GR}
$$
\Gamma_q(x):=\frac{(q;q)_\infty}{(q^x;q)_\infty}(1-q)^{1-x} \stackreb{\to}{q\to 1^-} \Gamma(x),
$$
where $\Gamma(x)$ is the standard Euler gamma function, we come to the following statement.
\begin{theorem}
Let $r\in\mathbb{Z}_{>1}$ and the variables $u_a\in\mathbb{C},\, a=1,\ldots, 5,$ are such that
$\textup{Re}(u_a)>0$. Then the following plain hypergeometric integral evaluation holds true
\begin{eqnarray} \nonumber &&
\frac{1}{2\pi \textup{i}}
\int_{-\textup{i}\infty}^{\textup{i}\infty}
\frac{\prod_{a=1}^3\Gamma(\frac{r-1}{r}+u_a+x,u_a-x)
\prod_{a=4}^5 \Gamma(u_a+x,\frac{1}{r}+u_a-x)}
{\Gamma(\frac{r-1}{r}+U+x,U-x,\frac{r-1}{r}+2x,\frac{1}{r}-2x)}dx
\\ && \makebox[2em]{}
=\frac{\prod_{1\leq a<b\leq 3}\Gamma(\frac{r-1}{r}+u_a+u_b)
\prod_{a=1}^3\prod_{b=4}^5\Gamma(u_a+u_b)\, \Gamma(\frac{1}{r}+u_4+u_5)}
{\prod_{a=1}^3\Gamma(U-u_a)\prod_{a=4}^5\Gamma(\frac{r-1}{r}+U-u_a)},
\label{NEWq=1}\end{eqnarray}
where $U=\sum_{a=1}^5u_a$.
\end{theorem}
A similar formula is obtained from \eqref{NEW} after flipping
$q$ and $q^{r-1}$ -- it corresponds to flipping $1/r$ and $(r-1)/r$ in
the arguments of gamma functions of \eqref{NEWq=1}.

Probably for other admissible values of discrete variables $n_a$ the limit $p\to 0$
produces other relations similar to \eqref{NEW} containing powers $q^k,\, k=2,\dots, r-2$.
We do not consider such possibilities here, since the derived result requires
a systematic consideration of not only all possible exact integral evaluations emerging in this way
or symmetry transformations (which we skip), but an investigation of the question what kind of
orthogonal polynomials and biorthogonal rational functions correspond to the measures
described by the above exactly computable $q$- and ordinary hypergeometric beta integrals and
their multivariate extensions. We conclude that the rarefied elliptic hypergeometric functions
contain intriguingly new phenomena at the lower hypergeometric levels.

\section{A $C_n$ rarefied elliptic beta integral of type II}

A rarefied analogue of the computable type II $C_n$-integral of \cite{vDS1}
has the following form. For convenience we denote the rank of the root system
(and the dimension of the related integral) as $d$, i.e. we consider the
root system $C_d$.

\begin{theorem}
Let nine continuous parameters $t, t_a (a=1,\ldots , 6), p, q \in \mathbb{C}^*$
and eight discrete variables $n,n_a\in\mathbb{Z},\, \epsilon=0,1,$ satisfy the
constraints $|p|, |q|,$ $|t|,$ $|t_a| <1$ and the balancing condition
\begin{equation}
t^{2d-2}\prod_{a=1}^6t_a=pq, \qquad 2n (d-1)+\sum_{a=1}^6n_a+3\epsilon=0.
\label{balII}\end{equation}
 Then
\begin{eqnarray}\nonumber &&\makebox[0em]{}
\kappa_d^{(r)}\sum_{m_1,\ldots,m_d=0}^{r-1}\int_{\T^d} \prod_{1\leq j<k\leq d}
\frac{\Gamma (tz_j^{\pm 1} z_k^{\pm 1},n\pm (m_j+\epsilon/2)\pm (m_k+\epsilon/2))}
{\Gamma (z_j^{\pm 1} z_k^{\pm 1},\pm (m_j+\epsilon/2)\pm (m_k+\epsilon/2))}
\\  \nonumber && \makebox[4em]{} \times
\prod_{j=1}^d\frac{\prod_{a=1}^6\Gamma (t_az_j,n_a +m_j+\epsilon)\Gamma(t_az_j^{-1},n_a- m_j)}
{\Gamma (z_j^{\pm2},\pm (2m_j+\epsilon))}
\frac{dz_j}{z_j}
\\  && \makebox[0em]{}
= \prod_{j=1}^d\left(\frac{\Gamma (t^j, nj)}{\Gamma (t,n)}
\prod_{1\leq a<b\leq 6}\Gamma (t^{j-1}t_at_b, n(j-1)+n_a+n_b+\epsilon)\right).
\label{rfintCII}\end{eqnarray}
\end{theorem}

\begin{proof}
The general scheme of proving formula \eqref{rfintCII} will be the same
as in the original considerations of \cite{vDS2}.
We assume that the variables $t_6$ and $n_6$ are excluded with the help of
the balancing condition and denote the expression on the left-hand side of
\eqref{rfintCII} as $I_{d,\epsilon}^{(r)}(t,t_1,\ldots,t_5,n,n_1,\ldots,n_5)$.
Consider now the following sum of $(2d-1)$-tuple integrals
\begin{eqnarray}
\label{compintII}
&&
\kappa_d^{(r)}\kappa_{d-1}^{(r)}\sum_{m_1,\ldots,m_d=0}^{r-1}
\sum_{l_1,\ldots,l_{d-1}=0}^{r-1}\int_{\T^{2d-1}}\prod_{j=1}^d\frac{dz_j}{z_j}
\prod_{k=1}^{d-1}\frac{dw_k}{w_k}
\\ && \nonumber \makebox[2em]{} \times
\prod_{1\leq j<k\leq d}\frac{1}{\Gamma (z_j^{\pm 1} z_k^{\pm 1},\pm (m_j+\epsilon/2)
\pm (m_k+\epsilon/2))}
\\  \nonumber && \makebox[3em]{} \times
 \prod_{j=1}^n\frac{\prod_{a=0}^5\Gamma (t_az_j^{\pm1},n_a+\frac{\epsilon}{2}\pm (m_j+\frac{\epsilon}{2}))}
{\Gamma (z_j^{\pm2},\pm (2m_j+\epsilon))}
\\  \nonumber && \makebox[4em]{} \times
\prod_{j=1}^d\prod_{k=1}^{d-1}
\Gamma (t^{1/2}z_jw_k,n +\frac{\epsilon}{2}+\frac{\delta}{2}\pm(m_j+\frac{\epsilon}{2})
\pm( l_k +\frac{\delta}{2})
\\ && \makebox[4em]{} \times
\prod_{1\leq j<k\leq d-1}\frac{1}{\Gamma (w_j^{\pm 1} w_k^{\pm 1},
\pm (l_j+\frac{\delta}{2})\pm (l_k+\frac{\delta}{2}))}
\prod_{k=1}^{d-1}\frac{1}{\Gamma (w_k^{\pm2},\pm (2l_k+\delta))}
\nonumber \\ && \makebox[-1em]{} \times
\prod_{k=1}^{d-1}
\frac{\Gamma (w_k^{\pm1} t^{d-3/2}\prod_{a=1}^5t_a,
(2d-3)(n+\frac{\epsilon}{2}+\frac{\delta}{2})+\sum_{a=1}^5(n_a+\frac{\epsilon}{2})
\pm (l_k+\frac{\delta}{2}))}
{\Gamma (w_k^{\pm 1} t^{2d-3/2}\prod_{a=1}^5t_a,
(4d-3)(n+\frac{\epsilon}{2}+\frac{\delta}{2})+\sum_{a=1}^5(n_a+\frac{\epsilon}{2})
\pm (l_k+\frac{\delta}{2}))  }
\nonumber\end{eqnarray}
with $|t|, |t_a|<1$ ($a=0,\ldots,5$), $\epsilon=0,1,\, \delta =0,1$, and
$$
t^{d-1}\prod_{a=0}^5t_a=pq, \qquad (d-1)(2n+\epsilon+\delta)
+\sum_{a=0}^5n_a+3\epsilon=0.
$$
Integration over the variables $w_k$ and summation over $l_k$ with the help
of formula \eqref{rfintCI} brings expression \eqref{compintII} to the form
$$
\frac{\Gamma (t,2n+\epsilon + \delta)^d}{\Gamma (t^d,(2n+\epsilon + \delta)d)}\;
I_{d,\epsilon}^{(r)}(t,t_1,\ldots,t_5,2n+\epsilon + \delta,n_1,\ldots,n_5),
$$
where the balancing condition \eqref{balII} is assumed with $n$ replaced
by $2n+\epsilon + \delta$.

Because the integrand is bounded on the contour of integration,
we can change the order of integrations. Then the integration over
$z_k$-variables and summation over $m_k$ with the help of formula \eqref{rfintCI}
converts expression \eqref{compintII} to
\begin{eqnarray*} &&
\Gamma (t,2n+\epsilon + \delta)^{d-1} \prod_{0\leq a< b\leq 5} \Gamma (t_at_b,n_a+n_b+ \epsilon)\,
\\ && \makebox[2em]{} \times
I_{d-1,\delta}^{(r)}(t,t^{1/2}t_1,\ldots,t^{1/2}t_5,2n+\epsilon + \delta,
n+n_1+\epsilon,\ldots,n+n_5+\epsilon).
\end{eqnarray*}
As a result, we obtain a recurrence relation connecting $I_{d,\epsilon}^{(r)}$-functions
of different rank $d$ and different $\epsilon$-variables:
\begin{eqnarray*} &&
\frac{I_{d,\epsilon}^{(r)}(t,t_1,\ldots,t_5,2n+\epsilon + \delta,n_1,\ldots,n_5)}
{I_{d-1,\delta}^{(r)}(t,t^{1/2}t_1,\ldots,t^{1/2}t_5,2n+\epsilon
+ \delta,n+n_1+\epsilon,\ldots,n+n_5+\epsilon)}
\\ && \makebox[2em]{}
= \frac{\Gamma (t^d,(2n+\epsilon + \delta)d)}{\Gamma (t,2n+\epsilon + \delta)}\makebox[-0.5em]{}
\prod_{0\leq a<b\leq 5}\makebox[-0.5em]{}\Gamma (t_at_b,n_a+n_b+ \epsilon).
\end{eqnarray*}
For different choices of $\epsilon$ and $\delta$  one can reach both
even and odd values of the sum $2n+\epsilon + \delta$, which can be redenoted as an arbitrary
integer $n\in\mathbb{Z}$. Then, using known initial condition for $d=1$ \eqref{rfint},
 we find \eqref{rfintCII} by recursion.
\end{proof}

Evidently, for $d=1$ both multiple integrals \eqref{rfintCI}
and \eqref{rfintCII} reduce to the rarefied elliptic beta integral \eqref{rfint}.
The relation \eqref{rfintCII} represents currently the most complicated
known generalization of the Selberg integral along the lines
introduced by Gustafson \cite{gust1}.

\section{An analogue of the Euler-Gauss hypergeometric function}

For the same values of the discrete variables $\epsilon$ and $\delta$ as
in the previous section, consider the following double sum of double integrals
\baq &&  \nonumber
\sum_{m=0}^{r-1}\sum_{l=0}^{r-1}
 \int_{\T^2}\Gamma (fz^{\pm1} w^{\pm1},h+\frac{\epsilon}{2}+\frac{\delta}{2}
 \pm (m+\frac{\epsilon}{2})\pm (l+\frac{\delta}{2}))
\\ && \makebox[2em]{} \times
\frac{\prod_{a=1}^4\Gamma (t_az^{\pm1}, n_a+\frac{\epsilon}{2}\pm (m+\frac{\epsilon}{2}))
\Gamma ( s_aw^{\pm1},k_a+\frac{\delta}{2}\pm (l+\frac{\delta}{2}))}
{\Gamma (z^{\pm 2},\pm (2m+\epsilon))\Gamma (w^{\pm 2},\pm (2l+\delta))} \frac{dz}{z}\frac{dw}{w},
\nonumber \eaq
where the variables $n_a,k_a,h\in \mathbb{Z}$ and $t_a,s_a,f\in \mathbb{C}^*$
satisfy the constraints $|t_a|, |s_a|,|f|, <1$ and the balancing conditions
\beq
f^2\prod_{a=1}^4t_a= f^2\prod_{a=1}^4s_a=pq, \qquad
2h+\epsilon+\delta=-\sum_{a=1}^4(n_a+\frac{\epsilon}{2})=-\sum_{a=1}^4 (k_a+\frac{\delta}{2}).
\label{bal1}\eeq
We see that for different values of $\epsilon$ and $\delta$,
the sums $\sum_{a=1}^4n_a$ and $\sum_{a=1}^4 k_a$ can be both (simultaneously) odd or even integers.

Because of the imposed constraints, the contour $\T$ is legitimate
for computing sum-integrals over $(z,m)$ or $(w,\ell)$ with the help
of formula \eqref{rfint}. Integrate first over $z$ and sum over $m$.
Then, using Fubini's theorem, we change the order and integrate first over $w$
and sum over $l$. This yields the following identity.
\begin{eqnarray}\label{trafo} && \makebox[-2em]{}
\sum_{l=0}^{r-1}\int_{\T} \frac{\prod_{a=1}^4\Gamma (ft_aw^{\pm1},h+n_a+\epsilon+\frac{\delta}{2}
\pm (l+\frac{\delta}{2}))
\Gamma (s_aw^{\pm1},k_a+\frac{\delta}{2}\pm (l+\frac{\delta}{2}))}{\Gamma (w^{\pm 2},\pm (2l+\delta))}
\frac{dw}{w}
 \\  \nonumber && \makebox[-1em]{}
= \prod_{1\leq a<b\leq 4}\frac{\Gamma (s_as_b,k_a+k_b+\delta)}
{\Gamma (t_at_b,n_a+n_b+\epsilon)}
\\ \nonumber && \makebox[0em]{} \times
\sum_{m=0}^{r-1}
\int_{\T} \frac {\prod_{a=1}^4\Gamma (t_az^{\pm1},n_a+\frac{\epsilon}{2}\pm (m+\frac{\epsilon}{2}))
\Gamma (fs_az^{\pm1},h+k_a+\delta+\frac{\epsilon}{2}\pm (m+\frac{\epsilon}{2}))}
{\Gamma (z^{\pm 2},\pm (2m+\epsilon))} \frac{dz}{z}.
\nonumber\end{eqnarray}

Let us define the rarefied elliptic hypergeometric function
\beq
V_\epsilon^{(r)}(t_a,n_a;p,q)=\kappa^{(r)}
\sum_{m=0}^{r-1}\int_{\T}
\frac {\prod_{a=1}^8\Gamma (t_az,n_a+\epsilon+ m)\Gamma (t_az^{-1},n_a- m)}
{\Gamma (z^{\pm 2},\pm (2m+\epsilon))} \frac{dz}{z},
\label{rfV}\eeq
where $t_a\in \mathbb{C}^*\, ,|t_a|<1$, $n_a\in\mathbb{Z}\, , \epsilon=0,1$, and
\beq
\prod_{a=1}^8 t_a=(pq)^2, \qquad \sum_{a=1}^8n_a+4\epsilon=0.
\label{bal2}\eeq
As usual, the contour $\T$ separates geometric progressions of poles
converging to zero from their partners going to infinity.
Other domains of values of the parameters are reached
by analytic continuation.
Both, the function itself \eqref{rfV} and the balancing condition \eqref{bal2},
are invariant with respect to the group $S_8$ permuting parameters $t_a$
and $n_a$ (which is the Weyl group of the root system $A_7$).
For $r=1$ this is the elliptic analogue
of the Euler-Gauss hypergeometric function introduced in \cite{spi:theta2},
$V_\epsilon^{(1)}(t_a,n_a;p,q)\equiv V(t_a;p,q)$.

Suppose that parameters $t_7,t_8, n_7,$ and $n_8$ satisfy the constraints
$t_7t_8=pq$ and $n_7+n_8+\epsilon=0$. Then we have
$$
\Gamma^{(r)}(t_7z,n_7+m+\epsilon)\Gamma^{(r)}(t_8z^{-1},n_8-m)=1
$$
and these parameters drop out completely from the $V^{(r)}$-function,
which thus becomes equal to the rarefied elliptic beta integral.

Quasiperiodicity of the rarefied elliptic gamma function leads
to the relation
\begin{eqnarray}\nonumber 
&&
V_\epsilon^{(r)}(\ldots, n_b+r,\ldots, n_c-r,\ldots;p,q)
=V_\epsilon^{(r)}(t_a,n_a;p,q)
\\ && \makebox[2em]{} \times
\left[ t_b^{r+2n_b+\epsilon} t_c^{r-2n_c-\epsilon}(p^{1-n_b-n_c-\epsilon}
q^{1+n_b+n_c+\epsilon})^{n_c-n_b-r}
\right]^{r-1}.
\nonumber\end{eqnarray}
For odd $r$ this relation allows one to convert all $n_a$ to even numbers, i.e.
without loss of generality, for odd $r$ we can assume that all $n_a$ are even.

Substituting definition \eqref{rfV} into relation \eqref{trafo}, we obtain
the transformation property of the $V_\epsilon^{(r)}$-function
\begin{eqnarray}\label{E7trafo1} &&  \makebox[-1em]{}
V_\epsilon^{(r)}(t_a,n_a;p,q)=V_\delta^{(r)}(s_a,k_a;p,q)
\\ && \makebox[2em]{} \times
\prod_{1\leq b<c\leq 4}\Gamma (t_bt_c,n_b+n_c+\epsilon)
\Gamma (t_{b+4}t_{c+4},n_{b+4}+n_{c+4}+\epsilon),
\nonumber\end{eqnarray}
where $$
\left\{
\begin{array}{cl}
s_a =f t_a,&  a=1,2,3, 4,  \\
s_a = f^{-1} t_a, &  a=5,6,7, 8,
\end{array}
\right.;
\quad f=\sqrt{\frac{pq}{t_1t_2t_3t_4}}=\sqrt{\frac{t_5t_6t_7t_8}{pq}},
$$
\begin{equation}
\left\{
\begin{array}{cl}
k_a= n_a-\frac{1}{2}(\sum_{b=1}^4n_b+\epsilon+\delta), &a=1,2,3, 4,  \\
k_a= n_a-\frac{1}{2}(\sum_{b=5}^8n_b+\epsilon+\delta), &  a=5,6,7, 8,
\end{array}
\right.
\label{ntok}\end{equation}
and is it assumed that $|t_a|, |s_a|<1$.
In the space of continuous parameters $t_a$, the map $t_a\to s_a$ is
the key reflection transformation extending $S_8$ to $W(E_7)$,
the Weyl group of the exceptional root system $E_7$, in the same way as
in the $r=1$ case.

However, in the space of discrete variables $n_a$ the situation is
more complicated. The function $V_\epsilon^{(r)}$ does not depend
on $\delta$, which has the appearance of a free parameter. However, this
is not true -- the above transformation is meaningful only when
$\sum_{b=1}^4n_b+\epsilon+\delta$ is an even integer, and it is from this condition
that the value of $\delta$ is determined. So, for even or odd
$\sum_{b=1}^4n_b+\epsilon$ one should take $\delta=0$ or
$\delta=1$, respectively. As mentioned above, for odd $r$ all $n_a$
can be taken even, and in this case one has $\delta=\epsilon$.
The same situation holds when all $n_a$ are odd. However,
when $r$ is even and variables $n_a$ take both odd and even integer values,
the transformation law of discrete variables does not look
$S_8$ symmetric, and $\delta\neq\epsilon$ cases are allowed.
In terms of the variables $n_a'=n_a+\epsilon/2$ and $k_a'=k_a+\delta/2$
the transformation \eqref{ntok} takes the standard reflection form
$$
k_a'= n_a'-\frac{1}{2}\sum_{b=1}^4n_b',\, a=1,2,3,4,  \quad
k_a'= n_a'-\frac{1}{2}\sum_{b=5}^8n_b',\,   a=5,6,7,8.
$$
Note that $n_a'$ and $k_a'$ may be not integer valued and the choice of
$\delta$ depends on $n_a$. Therefore, the action of full $W(E_7)$
symmetry on discrete variables $n_a$ gets a curious deformed
form. The detailed analysis of
this phenomenon lies beyond the scope of this work.

Similar to the standard $V$-function situation, there
are two more distinguished forms of the $W(E_7)$-transformations.
The second transformation is obtained after repeating \eqref{E7trafo1}
with $\delta$ playing the role of $\epsilon$,
$s_{3,4,5,6}$ playing the role of  $t_{1,2,3,4}$ and
$k_{3,4,5,6}$ playing the role of $n_{1,2,3,4}$. Also, one has to introduce
another discrete variable $\rho=0,1$, an analogue of $\delta$ in the first
transformation.
After symmetrization of
the resulting relation, we obtain the identity:
\begin{eqnarray}\label{E7trafo2} && 
V_\epsilon^{(r)}(t_a,n_a;p,q)
= \prod_{1\leq b,c\leq 4}\Gamma (t_bt_{c+4},n_b+n_{c+4}+\epsilon)\,
V_\rho^{(r)}\left(s_a,k_a;p,q\right),
\end{eqnarray}
where
$$
s_a= \frac{\sqrt{ t_1t_2t_3t_4 }}{t_a}, \; a=1,2,3,4,  \quad
s_a=  \frac{\sqrt{ t_5t_6t_7t_8 }}{t_a} , \; a=5,6,7,8,
$$
with $|t_a|, |s_a|<1$ and
$$
\left\{
\begin{array}{cl}
k_a= -n_a+\frac{1}{2}\left(\sum_{\ell=1}^4 n_\ell+\epsilon-\rho\right), & a=1,2,3,4,  \\
k_a= -n_a+\frac{1}{2}\left(\sum_{\ell=5}^8 n_\ell+\epsilon-\rho\right), & a=5,6,7,8.
\end{array}
\right.
$$
Here the value of $\rho$ is fixed from the condition that
$\sum_{\ell=1}^4 n_\ell+\epsilon-\rho$ is even.

The third transformation is obtained after equating the right-hand side
expressions in \eqref{E7trafo1} and \eqref{E7trafo2}:
\begin{eqnarray}\label{E7trafo3} && 
V_\epsilon^{(r)}(t_a,n_a;p,q)
= \prod_{1\leq b<c\leq 8}\Gamma (t_bt_c,n_b+n_c+\epsilon)\,
V_\epsilon^{(r)}\left(\frac{\sqrt{pq}}{t_a},-n_a-\epsilon;p,q\right).
\end{eqnarray}
Originally one finds in this case the transformation
$n_a\to -n_a-(\epsilon+\delta)/2$, which requires $\epsilon+\delta$ to
be even leading to $\delta=\epsilon$. Note that the form of this
third transformation is true for arbitrary values of $n_a$,
since the map $n_a\to -n_a-\epsilon$ is true for all $n_a$.

For $r=1$ all three relations become the standard symmetry
transformations for the $V$-function with the key
generating relation \eqref{E7trafo1} discovered in \cite{spi:theta2}.

Let us construct the rarefied analogue of the elliptic hypergeometric equation
derived in \cite{spi:thesis,spi:tmf2007}. For brevity we use the following trick --
until formula \eqref{reheq} the symbols $n_a$ and $m$ actually denote
$n_a+\epsilon/2$ and $m+\epsilon/2$. This is legitimate, since
the integrality of $n_a$ and $m$ is not essential in the computations.
However, to remind on the $\epsilon$-dependence we keep the notation $V_\epsilon^{(r)}$.

The addition formula for elliptic theta functions can be written in the form
\beq
t_3\theta(t_2t_3^{\pm1},t_1z^{\pm1};q^r)+t_1\theta(t_3t_1^{\pm1},t_2z^{\pm1};q^r)
+t_2\theta(t_1t_2^{\pm1},t_3z^{\pm1};q^r)=0.
\label{add}\eeq
It yields the following contiguous relation for the $V_\epsilon^{(r)}$-function
\baq \nonumber &&
\frac{t_1^{1+2n_1}q^{-n_1(n_1+2)} V_\epsilon^{(r)}(pt_1,n_1-1)}
{\theta(t_1t_2^{\pm1}q^{-n_1\mp n_2},t_1t_3^{\pm 1}q^{-n_1\mp n_3};q^r)}
+\frac{t_2^{1+2n_2}q^{-n_2(n_2+2)}V_\epsilon^{(r)}(pt_2,n_2-1)}
{\theta(t_2t_1^{\pm1}q^{-n_2\mp n_1},t_2t_3^{\pm 1}q^{-n_2\mp n_3};q^r)}
\\ && \makebox[6em]{}
+\frac{t_3^{1+2n_3}q^{-n_3(n_3+2)}V_\epsilon^{(r)}(pt_3,n_3-1)}
{\theta(t_3t_1^{\pm1}q^{-n_3\mp n_1},t_3t_2^{\pm 1}q^{-n_3\mp n_2};q^r)} = 0,
\label{cont-1}\eaq
where $V_\epsilon^{(r)}(pt_b,n_b-1)$ denotes the $V_\epsilon^{(r)}(t_a,n_a)$-function
with the parameters $t_b, n_b$ replaced by $pt_b, n_b-1$ (with the balancing condition
being $\prod_{a=1}^8t_a=pq^2,\, \sum_{a=1}^8n_a=1$).
Indeed, if we replace in \eqref{cont-1} $V_\epsilon^{(r)}$-functions by
their integrands, then we obtain the equality
\begin{eqnarray*} &&
\frac{t_1q^{-n_1}\theta(t_1z^{\pm1}q^{-n_1\mp m};q^r)}
{\theta(t_1t_2^{\pm1}q^{-n_1\mp n_2},t_1t_3^{\pm 1}q^{-n_1\mp n_3};q^r)}
+\frac{t_2q^{-n_2}\theta(t_2z^{\pm1}q^{-n_2\mp m};q^r)}
{\theta(t_2t_1^{\pm1}q^{-n_2\mp n_1},t_2t_3^{\pm 1}q^{-n_2\mp n_3};q^r)}
\\ && \makebox[4em]{}
+\frac{t_3q^{-n_3}\theta(t_3z^{\pm1}q^{-n_3\mp m};q^r)}
{\theta(t_3t_1^{\pm1}q^{-n_3\mp n_1},t_3t_2^{\pm 1}q^{-n_3\mp n_2};q^r)} = 0
\end{eqnarray*}
multiplied by the function
$$
z^{2m}q^{m^2}\prod_{a=1}^8\frac{\Gamma (t_az^{\pm1},n_a\pm m)}
{\Gamma (z^{\pm2},\pm 2m)}.
$$
Replacing $t_a\to t_aq^{n_a}$ and $z\to zq^{m}$
and simplifying the factors we obtain the addition formula \eqref{add}.
Integrating the
resulting equation for the integrand functions over $z\in\T$
and summing in $m$, we come to \eqref{cont-1}. Note that for $r=1$ one can
pull out all powers of $q$ out of the theta functions and find that
all three terms in \eqref{cont-1} get equal multipliers
$q^{\sum_{a=1}^3n_a^2}/\prod_{a=1}^3t_a^{2n_a}$, so that
the dependence of this contiguous relation on $n_a$ disappears completely.

Substituting relation \eqref{E7trafo3} in \eqref{cont-1}, we obtain
\baq \nonumber &&
\frac{p^{2n_1}q^{n_1(n_1+3)}}{t_1^{2+2n_1}}
\frac{\prod_{a=4}^8\theta\left(\frac{t_1t_a}{pq}q^{-n_1-n_a};q^r\right)}
{\theta(\frac{t_2}{t_1}q^{n_1-n_2},\frac{t_3}{t_1}q^{n_1-n_3};q^r)}V_\epsilon^{(r)}(p^{-1}t_1,n_1+1)
\\   \nonumber && \makebox[2em]{}
+\frac{p^{2n_2}q^{n_2(n_2+3)}}{t_2^{2+2n_2}}
\frac{\prod_{a=4}^8\theta\left(\frac{t_2t_a}{pq}q^{-n_2-n_a};q^r\right)}
{\theta(\frac{t_1}{t_2}q^{n_2-n_1},\frac{t_3}{t_2}q^{n_2-n_3};q^r)}V_\epsilon^{(r)}(p^{-1}t_2,n_2+1)
\\ && \makebox[4em]{}
+\frac{p^{2n_3}q^{n_3(n_3+3)}}{t_3^{2+2n_3}}
\frac{\prod_{a=4}^8\theta\left(\frac{t_3t_a}{pq}q^{-n_3-n_a};q^r\right)}
{\theta(\frac{t_2}{t_3}q^{n_3-n_2},\frac{t_1}{t_3}q^{n_3-n_1};q^r)}V_\epsilon^{(r)}(p^{-1}t_3,n_3+1)
=0,
\label{cont-3}\eaq
where $\prod_{a=1}^8t_a=p^3q^2$ and $\sum_{a=1}^8n_a=-1$.
Shifting $t_3\to pt_3$ and $n_3\to n_3-1$ in \eqref{cont-3}, we come to the equality
\baq \nonumber &&
\frac{p^{2n_1}q^{n_1(n_1+3)}}{t_1^{2+2n_1}}
\frac{\prod_{a=4}^8\theta\left(\frac{t_1t_a}{pq}q^{-n_1-n_a};q^r\right)}
{\theta(\frac{t_2}{t_1}q^{n_1-n_2},\frac{pqt_3}{t_1}q^{n_1-n_3};q^r)}
V_\epsilon^{(r)}(p^{-1}t_1,pt_3,n_1+1,n_3-1)
\\   \nonumber && \makebox[2em]{}
+\frac{p^{2n_2}q^{n_2(n_2+3)}}{t_2^{2+2n_2}}
\frac{\prod_{a=4}^8\theta\left(\frac{t_2t_a}{pq}q^{-n_2-n_a};q^r\right)}
{\theta(\frac{t_1}{t_2}q^{n_2-n_1},\frac{pqt_3}{t_2}q^{n_2-n_3};q^r)}
V_\epsilon^{(r)}(p^{-1}t_2,pt_3,n_2+1,n_3-1)
\\ && \makebox[4em]{}
+\frac{p^{2(n_3-1)}q^{(n_3-1)(n_3+2)}}{(pt_3)^{2n_3}}
\frac{\prod_{a=4}^8\theta\left(t_3t_aq^{-n_3-n_a};q^r\right)}
{\theta(\frac{t_2}{pqt_3}q^{n_3-n_1},\frac{t_1}{pqt_3}q^{n_3-n_1};q^r)}
V_\epsilon^{(r)}(t_a,n_a) =0.
 \nonumber\eaq
Replacing $t_1\to p^{-1}t_1, \, n_1\to n_1+1$ or
$t_2\to p^{-1}t_2, \, n_2\to n_2+1$ in \eqref{cont-1} we obtain the relations
\baq \nonumber && \makebox[0em]{}
\frac{(p^{-1}t_1)^{3+2n_1}q^{-(n_1+1)(n_1+3)} V_\epsilon^{(r)}(t_a,n_a)}
{\theta(p^{-1}t_1t_2^{\pm1}q^{-n_1-1\mp n_2},p^{-1}t_1t_3^{\pm 1}q^{-n_1-1\mp n_3};q^r)}
\\ && \makebox[4em]{}
+\frac{t_2^{1+2n_2}q^{-n_2(n_2+2)}V_\epsilon^{(r)}(p^{-1}t_1,pt_2,n_1+1,n_2-1)}
{\theta(t_2(p^{-1}t_1)^{\pm1}q^{-n_2\mp (n_1+1)},t_2t_3^{\pm 1}q^{-n_2\mp n_3};q^r)}
\nonumber \\ && \makebox[6em]{}
+\frac{t_3^{1+2n_3}q^{-n_3(n_3+2)}V_\epsilon^{(r)}(p^{-1}t_1,pt_3,n_1+1,n_3-1)}
{\theta(t_3(p^{-1}t_1)^{\pm1}q^{-n_3\mp (n_1+1)},t_3t_2^{\pm 1}q^{-n_3\mp n_2};q^r)} = 0
 \nonumber\eaq
or
\baq \nonumber && \makebox[0em]{}
\frac{t_1^{1+2n_1}q^{-n_1(n_1+2)} V_\epsilon^{(r)}(pt_1,p^{-1}t_2,n_1-1,n_2+1)}
{\theta(t_1(p^{-1}t_2)^{\pm1}q^{-n_1\mp (n_2+1)},t_1t_3^{\pm 1}q^{-n_1\mp n_3};q^r)}
\\ && \makebox[4em]{}
+\frac{(p^{-1}t_2)^{3+2n_2}q^{-(n_2+1)(n_2+3)}V_\epsilon^{(r)}(t_a,n_a)}
{\theta(p^{-1}t_2t_1^{\pm1}q^{-n_2-1\mp n_1},p^{-1}t_2t_3^{\pm 1}q^{-n_2-1\mp n_3};q^r)}
\nonumber \\ && \makebox[6em]{}
+\frac{t_3^{1+2n_3}q^{-n_3(n_3+2)}V_\epsilon^{(r)}(p^{-1}t_2,pt_3,n_2+1,n_3-1)}
{\theta(t_3t_1^{\pm1}q^{-n_3\mp n_1},t_3(p^{-1}t_2)^{\pm 1}q^{-n_3\mp (n_2+1)};q^r)}
 = 0.
 \nonumber\eaq
Eliminating from the latter three equalities the functions
$V_\epsilon^{(r)}(p^{-1}t_1,pt_3,n_1+1,n_3-1)$ and $V_\epsilon^{(r)}(p^{-1}t_2,pt_3,n_2+1,n_3-1)$,
we arrive at the final equation (we restore the $\epsilon$-dependence):
\begin{eqnarray}\nonumber
&& \makebox[0em]{}
\mathcal{A}\left(\textstyle{\frac{t_1}{q^{n_1+\epsilon/2}},\frac{t_2}{q^{n_2+\epsilon/2}},
\ldots,\frac{t_8}{q^{n_8+\epsilon/2}},p;q^r}\right)
\Big(U_\epsilon(pt_1,p^{-1}t_2,n_1-1,n_2+1)-U_\epsilon(t_a,n_a)\Big)
\\ \nonumber && \makebox[1em]{}
+\mathcal{A}\left(\textstyle{\frac{t_2}{q^{n_2+\epsilon/2}},\frac{t_1}{q^{n_1+\epsilon/2}},
\ldots,\frac{t_8}{q^{n_8+\epsilon/2}},p;q^r}\right)
\Big(U_\epsilon(p^{-1}t_1,pt_2,n_1+1,n_2-1)-U_\epsilon(t_a,n_a)\Big)
\\ && \makebox[4em]{}
+ U_\epsilon(t_a,n_a)=0,
\label{reheq}\end{eqnarray}
where we have denoted
\begin{equation}
 \mathcal{A}(t_1,\ldots, t_8,p;q^r):=\frac{\theta\left(\frac{t_1}{pq^{1-r}t_3},
t_3t_1,\frac{t_3}{t_1};q^r\right)}
                 {\theta\left(\frac{t_1}{t_2},
\frac{t_2}{pq^{1-r}t_1},\frac{t_1t_2}{pq^{1-r}};q^r\right)}
\prod_{a=4}^8\frac{\theta\left(\frac{t_2t_a}{pq^{1-r}};q^r\right)}
{\theta\left(t_3t_a;q^r\right)}
\label{A-coeff}\end{equation}
and
$$
U_\epsilon(t_a,n_a):=\frac{V_\epsilon^{(r)}(t_a,n_a)}
{\prod_{k=1}^2\Gamma (t_kt_3,n_k+n_3+\epsilon)\Gamma (t_kt_3^{-1},n_k- n_3)}.
$$
A fundamental fact is that for any $r$ the function $\mathcal{A}(t_1,\ldots, t_8,p;q^r)$
is a $q^r$-elliptic function of all parameters $t_1,\ldots,t_8$ (one of which should be
counted as a dependent variable through the balancing condition $\prod_{a=1}^8t_a=(pq)^2$),
i.e. it does not change after the scaling $t_a\to t_aq^r, t_b\to t_bq^{-r}$ for any $a\neq b$.

We call equation \eqref{reheq} the {\em rarefied elliptic hypergeometric equation},
though it does not have the form one would have liked to see.
It can be checked that under the shifts $n_a\to n_a+r,\, n_b\to n_b-r$, $a\neq b$
the functions $U_\epsilon(p^{\pm1}t_1,p^{\mp1}t_2,n_1\mp1,n_2\pm 1)$ and $U_\epsilon(t_a,n_a)$ have the same
quasiperiodicity multipliers.
Therefore, by shifting $t_{1,2}\to p^{\pm l}t_{1,2}$, $n_{1,2}\to n_{1,2} \mp l$, $l=1,2,\ldots,$
in equation \eqref{reheq}, combining the resulting equation in an appropriate way and
using the fact that
$$
U_\epsilon(t_a,n_1- r,n_2+ r,n_3,\ldots) =U_\epsilon(t_a,n_1+ r,n_2- r,n_3,\ldots) = U_\epsilon(t_a,n_a),
$$
one can derive the following tridiagonal equation
\begin{equation}
\alpha(t_a,n_a)U_\epsilon(p^rt_1,p^{-r}t_2,n_a)+\beta(t_a,n_a)U_\epsilon(t_a,n_a)+\gamma(t_a,n_a)U_\epsilon(p^{-r}t_1,p^{r}t_2,n_a)=0,
\label{true-reheq}\end{equation}
for some coefficients $\alpha,\beta,\gamma$. After parametrization $t_1=cx, t_2=cx^{-1}$
the latter equation becomes a $``q"$-difference equation of the second order for the
variable $x$ with $``q"=p^r$. It is appropriate
to call equation  \eqref{true-reheq} the rarefied elliptic hypergeometric equation, however,
we were not able to derive a compact form of its coefficients.
Note that we know already one of its solutions given by the $V_\epsilon^{(r)}$-function.
Its second linearly independent solution is obtained by application
of the symmetry transformation of the equation which is not a symmetry of the
solution, e.g. by multiplying its parameters by the powers of $q^r$ or some other means.
The second order finite-difference equation \eqref{true-reheq}
represents currently the most complicated known equation of such type with the closed form
solutions (an ``exactly solvable" equation).

Similar to the standard $r=1$ case, equation \eqref{reheq}
has a partner obtained by permuting the bases $p$ and $q$:
\begin{eqnarray}\nonumber
&& \makebox[-1em]{}
\mathcal{A}\left(t_1p^{n_1+\epsilon/2},t_2p^{n_2+\epsilon/2},\ldots,
t_8p^{n_8+\epsilon/2},q;p^r\right)
\Big(U_\epsilon(qt_1,q^{-1}t_2,n_1+1,n_2-1)-U_\epsilon(t_a,n_a)\Big)
\\ \nonumber && \makebox[-1em]{}
+\mathcal{A}\left(\textstyle{t_2p^{n_2+\epsilon/2},t_1p^{n_1+\epsilon/2},\ldots,
t_8p^{n_8+\epsilon/2},q;p^r}\right)
\Big(U_\epsilon(q^{-1}t_1,qt_2,n_1-1,n_2+1)-U_\epsilon(t_a,n_a)\Big)
\\ && \makebox[4em]{}
 + U_\epsilon(t_a,n_a)=0.
\label{reheq_par}\end{eqnarray}

Let us set $\epsilon=0$ and denote
$$
t_1:=cx, \quad t_2:=\frac{c}{x},\quad \text{or} \quad c=\sqrt{t_1t_2},
\quad x=\sqrt\frac{t_1}{t_2}
$$
and
$$
n_1:=n_c+n,\quad n_2:= n_c-n, \quad \text{or}\quad
n_c= \frac{n_1+n_2}{2}, \quad n= \frac{n_1-n_2}{2}.
$$
Now we introduce new continuous and discrete variables
$$
s_1:=\frac{c}{t_3pq^{1-r}},
\quad s_2:=\frac{c}{t_3},\quad s_3:=ct_3q^{4r}, \quad
s_a:=\frac{pq^{1-r}}{ct_a},\; a=4,\ldots,8,
$$
$$
k_1=k_2:=n_c-n_3, \quad k_3:=n_c+n_3, \quad k_a:=-n_c-n_a,\; a=4,\ldots,8.
$$
To keep integrality of the numbers $k_a$ we shall assume that all $n_a$
are either odd or even.
It is easy to see  that the balancing condition remains intact
$$
\prod_{a=1}^8s_a=p^2q^2, \qquad \sum_{a=1}^8 k_a=0.
$$
Replacing $U_\epsilon(t_a,n_a)$ by unknown function
$f(x,n)$, we obtain another form of the rarefied elliptic hypergeometric equation:
\baq \nonumber &&
 A(xq^{-n})\left( f(px,n-1)-f(x,n)\right)
 \\ && \makebox[2em]{}
+ A(x^{-1}q^n)\left( f(p^{-1}x,n+1)-f(x,n)\right) + \nu f(x,n)=0,
\label{reheq_sym}\eaq
where
\beq
A(x)=\frac{\prod_{a=1}^8 \theta\left(\frac{s_ax}{q^{k_a}};q^r\right) }
{\theta\left(x^2,pq^{1-r}x^2;q^r\right)},
\qquad
\nu=\prod_{a=3}^8\theta\left(\frac{s_1s_a}{q^{k_1+k_a}};q^r\right).
\label{pot}\eeq

\section{Extension of the multiple integral symmetry transformations}

We now consider rarefied analogues of the symmetry transformations for multiple
elliptic hypergeometric integrals derived by Rains in \cite{rai:trans}. Define the general
type I rarefied elliptic hypergeometric function on the root system $C_d$:
\begin{eqnarray}\nonumber &&
I_{C_d,\epsilon}^{(m)}(t_a,n_a):=\kappa^{(r)}_d
\sum_{m_1,\ldots,m_d=0}^{r-1}\int_{\T^d}
\prod_{1\leq j<k\leq d}
\frac{1}{\Gamma(z_j^{\pm 1}z_k^{\pm 1},\pm (m_j+\epsilon/2)\pm (m_k+\epsilon/2)))}
\\ && \makebox[4em]{} \times
\prod_{j=1}^n\frac{\prod_{a=1}^{2d+2m+4}
\Gamma(t_a z_j,n_a+ m_j+\epsilon)
\Gamma(t_a z_j^{- 1},n_a-m_j)}
{\Gamma(z_j^{\pm 2},\pm (2m_j+\epsilon))}\frac{dz_j}{z_j},
\label{rarI}\end{eqnarray}
where $|t_a|<1$ and the balancing condition has the form
\beq
\prod_{a=1}^{2d+2m+4}t_a=(pq)^{m+1},\qquad
\sum_{a=1}^{2d+2m+4}n_a+(d+m+2)\epsilon=0.
\label{balgenI}\eeq

\begin{conjecture}
Suppose that all $t_a\in\mathbb{C}^*$ and $n_a\in\mathbb{Z}$, $
a=1,\ldots,2d+2m+4$, $\epsilon=0,1$ satisfy the constraints
$\sqrt{|pq|}<|t_a|<1$ and the balancing condition \eqref{balgenI}. Then
\begin{eqnarray}  &&  \makebox[0em]{}
I_{C_d,\epsilon}^{(m)}(t_a,n_a)=
\prod_{1\leq a<b\leq 2d+2m+4}\Gamma (t_at_b,n_a+n_b+\epsilon)\;
I_{C_m,\epsilon}^{(n)}\left(\frac{\sqrt{pq}}{t_a},-n_a-\epsilon\right).
\label{E7trafoCdI} \end{eqnarray}
\end{conjecture}

For $r=1$ this is the Rains'  $C_d\leftrightarrow C_m$ transformation
for type I elliptic hypergeometric integrals (see Theorem 3.1 in \cite{rai:trans}).
For $m=0$ one gets evaluation of the type I integral \eqref{rfintCI}. Equivalently,
it can be obtained from the general formula \eqref{E7trafoCdI}
after fixing $n_{2d+4+a}+n_{2d+4+a+m}+\epsilon=0$ and taking the limit $t_{2d+4+a}t_{2d+m+4+a}\to pq$
for $a=1,\ldots,m.$ In this case the parameters $t_{2d+5},\ldots,t_{2d+2m+4}$
and $n_{2d+5},\ldots, n_{2d+2m+4}$ simply drop out from the
expression on the left-hand side, and on the right-hand side a number of
poles pinch the integration contours reducing sums of integrals to the right-hand
side expression in \eqref{rfintCI}. For $n=m=1$ we obtain the third
$V_\epsilon^{(r)}$-function transformation \eqref{E7trafo3}. For arbitrary $r$
and $\epsilon=0$ this transformation was recently established in \cite{KY}
in terms of the gamma function \eqref{pergamma}
(the $\text{mod}\; r$ relation in discrete balancing condition of  \cite{KY}
disappears after passing to our normalization of the rarefied elliptic gamma function).

Define now the type II $C_d$-root system analogue of the rarefied elliptic
hypergeometric function \eqref{rfV}. Take two bases $p,q\in\mathbb{C},\, |p|, |q|<1$,
and 19 continuous and discrete parameters $t,t_a \in\mathbb{C}^*$
and $n,n_a\in\mathbb{Z}$ $(a=1,\ldots,8)$, $\epsilon=0,1,$
and impose the balancing condition
\beq
t^{2d-2}\prod_{a=1}^8 t_a=(pq)^2, \qquad 2n(d-1)+\sum_{a=1}^8n_a+4\epsilon=0.
\label{bal_bcd}\eeq
The type II $C_d$-extension of the $V$-function has the form
\baq
V_\epsilon^{(r)}(t,t_a,n,n_a):=\kappa^{(r)}_d
\sum_{m_1,\ldots,m_r=0}^{r-1}\int_{\T^d} \Delta_\epsilon^{(r)} (z_k,m_k;t,t_a,n,n_a)\prod_{j=1}^d\frac{dz_j}{z_j},
\label{rfVCd}\eaq
where
\baq\nonumber &&
\Delta_\epsilon^{(r)}(z_k,m_k;t,t_a,n,n_a)=\prod_{1\leq j<k\leq d}
\frac{\Gamma (tz_j^{\pm 1} z_k^{\pm 1},n\pm (m_j+\frac{\epsilon}{2})
\pm (m_k+\frac{\epsilon}{2}))}
{\Gamma (z_j^{\pm 1} z_k^{\pm 1},\pm (m_j+\frac{\epsilon}{2})
\pm (m_k+\frac{\epsilon}{2}))}
\\ && \makebox[4em]{} \times
\prod_{j=1}^d\frac{\prod_{a=1}^8\Gamma (t_az_j,n_a+\epsilon+ m_j)
\Gamma (t_az_j^{- 1},n_a-m_j)}
{\Gamma (z_j^{\pm2},\pm (2m_j+\epsilon))}
\label{weightVd}\eaq
with $|t|, |t_a|<1$. For $r=1$ this is the function introduced by Rains
in \cite{rai:trans} (see the function $\mathord{I\!I}^{(n)}_{\lambda,\mu}$ on p. 224
for $\lambda=\mu=0$).

\begin{conjecture}
The following identity should hold true
\begin{eqnarray}\nonumber  &&  \makebox[-3em]{}
V_\epsilon^{(r)}(t,t_a,n,n_a)=
\prod_{1\leq a<b\leq 4}\prod_{l=0}^{d-1}\Gamma (t^lt_at_b,ln+n_a+n_b+\epsilon)
\\ && \makebox[-2em]{} \times
\Gamma (t^lt_{a+4}t_{b+4},ln+n_{a+4}+n_{b+4}+\epsilon)
V_\delta^{(r)}(t,s_a,k_a),
\label{E7trafoCd} \end{eqnarray}
where
$$
s_a =ft_a, \; a=1,2,3,4,  \quad
s_a = f^{-1} t_a, \;  a=5,6,7,8, \quad
f=\sqrt{\frac{pqt^{1-d}}{t_1t_2t_3t_4}},
$$
$$
\left\{
\begin{array}{cl}
k_a= n_a-\frac{1}{2}(\sum_{b=1}^4n_b+(d-1)n+\epsilon+\delta), &a=1,2,3, 4,  \\
k_a= n_a-\frac{1}{2}(\sum_{b=5}^8n_b+(d-1)n+\epsilon+\delta), &  a=5,6,7, 8,
\end{array}
\right.
$$
and $|t|, |t_a|, |s_a|<1$ together with the condition
that $\delta$ is fixed from the demand that
$\sum_{b=1}^4n_b+(d-1)n+\epsilon+\delta$ is even.
\end{conjecture}

For $d=1$ this is the relation \eqref{E7trafo1} proven above.
For $r=1$, $d>1$ this is Rains' identity \cite{rai:trans}
(see Theorem 9.7 for $\lambda=\mu=0$).
In the limit $t_7t_8\to pq$ and $n_7+n_8+\epsilon=0$ the left-hand side integral reduces to
the integral in \eqref{rfintCII}, whereas the right-hand side
integral should collapse to the required product of gamma functions due to
pinching of the integration contours.

The general rarefied elliptic hypergeometric function of type I for
the root system $A_n$ has the form
\begin{eqnarray} \label{rarIAn} &&
I_{A_n,\epsilon}^{(m)}(t_a,n_a;s_a,k_a):=
\frac{(p^r;p^r)_\infty^n (q^r;q^r)_\infty^n}{(n+1)!(2\pi \textup{i})^n }
\\ && \makebox[-1em]{} \times
\sum_{0\leq m_1,\ldots,m_n\leq r-1 \atop m_1+\cdots+m_{n+1}=\epsilon
}\int_{\T^n}\frac{
\prod_{j=1}^{n+1}\prod_{a=1}^{n+m+2}
\Gamma(t_a z_j,n_a+ m_j)\Gamma(s_a z_j^{-1},k_a-m_j)
}{\prod_{1\leq i<j\leq n+1}\Gamma(z_iz_j^{-1},m_i-m_j)\Gamma(z_i^{-1}z_j,-m_i+m_j)}
\prod_{j=1}^n\frac{dz_j}{z_j},
\nonumber\end{eqnarray}
where $\prod_{j=1}^{n+1}z_j=1$, $|t_a|, |s_a|<1$ and the balancing condition has the form
\beq
TS=(pq)^{m+1},\quad T:=\prod_{a=1}^{n+m+2}t_a,\quad
S:=\prod_{a=1}^{n+m+2}s_a,\quad
\sum_{a=1}^{n+m+2}(n_a+k_a)=0.
\label{balgenIAn}\eeq
Suppose that $\epsilon=l(n+1)$ for some integer $l$. Then $\epsilon$ can be
removed by simple shifts $m_a\to m_a+l$. Therefore the natural values
of the parameter $\epsilon$, when it cannot be removed in this way, are
$\epsilon=0,1,\ldots, n.$ The function \eqref{rarIAn} has been considered
recently in \cite{KY} in terms of the gamma function \eqref{pergamma}.

Suppose that all $t_a,\, s_a\in\mathbb{C}^*$ and $n_a\,\, k_a \in\mathbb{Z}$,
$a=1,\ldots,n+m+2$, $\epsilon=0,1,\ldots, n$ satisfy the constraints
$|t_a|, |T^{\frac{1}{m+1}}t_a^{-1}|, |s_a|, |S^{\frac{1}{m+1}}s_a^{-1}|<1$
and the balancing condition \eqref{balgenIAn}. Then one has the following symmetry transformation
\begin{eqnarray} \label{E7trafoAnI}  &&  \makebox[0em]{}
I_{A_n,\epsilon}^{(m)}(t_a,n_a;s_a,k_a) =
\prod_{1\leq a<b\leq n+m+2}\Gamma (t_as_b,n_a+k_b)\;
\\ && \makebox[2em]{} \times
I_{A_m,\delta}^{(n)}\left(\frac{T^{\frac{1}{m+1}}}{t_a},
N-n_a; \frac{S^{\frac{1}{m+1}}}{s_a},-N-k_a\right),
\quad N:=\frac{\sum_{b=1}^{n+m+2} n_b+\epsilon-\delta}{m+1},
\nonumber\end{eqnarray}
where the value of $\delta=0,1,\ldots, m$ is fixed from the condition that
$\sum_{b=1}^{n+m+2} n_b+\epsilon-\delta$ is divisible by $m+1$.

For $n=m=1$ this identity coincides with the second $V_\epsilon^{(r)}$-function
transformation \eqref{E7trafo2}, provided one substitutes $s_a:=t_{a+4}$ and
$k_a:=n_{a+4}-\epsilon$, $a=1,\ldots,4$.  For $r=1$, $n>1,\, m>0$ this is Rains'
$A_n\leftrightarrow A_m$ transformation (see Theorem 4.1 in \cite{rai:trans}).
A symmetry transformation for $I_{A_n,\epsilon}^{(m)}$-function
was suggested in \cite{KY}, but the original proposition contained a mistake,
which was corrected after the author proposed \eqref{E7trafoAnI}.
The final transformation given in \cite{KY} is equivalent
to \eqref{E7trafoAnI} and has a simpler form.

As to possible proof of the above Conjecture 1, the considerations of
\cite{KY} should be applicable to it as well. However, a substantially
more elegant approach would consist in the appropriate generalization
of the method suggested in \cite{RS}. The $m=0$ case of \eqref{E7trafoAnI}
should be easily provable by direct extension of the method of \cite{spi:short}
used there for the type I $A_n$ elliptic beta integral evaluation.
The only available at the moment possible approach to multivariate type II
$V_\epsilon^{(r)}$-function transformation consists in an appropriate
generalization of the rather complicated proof of the $r=1$ case
suggested by Rains \cite{rai:trans}. We do not dwell into these
considerations in the present work.

Next we consider possible applications of the type II $C_n$ rarefied
elliptic hypergeometric function. For simplicity we limit
ourself to the case $\epsilon=0$.
Similar to the situation investigated in \cite{spi:thesis,spi:tmf2007},
we consider the space of sequences of holomorphic functions of
$z_j\in\mathbb{C}^*$,  which are $r$-periodic in the
discrete variables,
$\varphi(z_j,\ldots, m_k,m_k+r,m_{k+1},\ldots)=\varphi(z_j,m_j)$,
and define the inner product for it
$$
\langle \varphi, \psi \rangle=\kappa_{d}^{(r)}
\sum_{m_1,\ldots,m_d=0}^{r-1}\int_{{\mathbb T}^d}\Delta_0^{(r)}(z_k,m_k;t,t_a,n,n_a)
\varphi(z_j,m_j)\psi(z_j,m_j)\,\prod_{k=1}^d\frac{dz_k}{z_k}
$$
with the weight function \eqref{weightVd}.
Let us introduce the finite-difference operator
\begin{eqnarray}\nonumber
&& {\mathcal D}=\sum_{j=1}^d\Big(A_j(z_kq^{-m_k})(T_{p,j}S_j^{-1}-1)
+A_j(z_k^{-1}q^{m_k})(T_{p,j}^{-1}S_j-1)\Big),
\\&&
 A_j(z_k)=\frac{\prod_{a=1}^8\theta(t_aq^{-n_a}z_j;q^r)}{\theta(z_j^2,pq^{1-r}z_j^2;q^r)}
\prod_{\substack{\ell=1 \\ \ell\ne j}}^d
\frac{\theta(tq^{-n}z_jz_\ell^{\pm1};q^r)}{\theta(z_jz_\ell^{\pm 1};q^r)},
\label{rvD}\end{eqnarray}
where $T_{p,j}^{n}S_j^{m}f(z_k,m_k)=f(\ldots,p^{n}z_j,\ldots,m_j+m,\ldots)$
and we assume validity of the balancing restriction \eqref{bal_bcd}.
For $r=1$ this is the Hamiltonian of the van Diejen completely integrable
model \cite{die:integrability} under the additional balancing condition.
Suppose that the parameters $t_a$ are constrained in such a way
that the unit circle $\T$ separates the sequences of poles converging to zero
$z_j=0$ in the expression $\langle \varphi, {\mathcal D}\psi \rangle$
from their partners going to infinity. Then the  operator \eqref{rvD}
is symmetric with respect to the above inner product,
$$
\langle \varphi, {\mathcal D}\psi \rangle
=\langle {\mathcal D}\varphi, \psi \rangle.
$$
Surprisingly, this statement requires a rather complicated computation
associated with the presence of the powers $q^r$ in the arguments of
theta functions. Because the suggested generalization of the van Diejen
operator does not touch its analytical structure (or, more precisely, does
not change the divisor structure of the functional coefficients entering it),
our operator should define a completely integrable quantum many body system as well
(in the sense that there exist $d$ commuting finite-difference operators of a similar
form of the higher order in the shifting operators $T_{p,j}S_j^{-1}$). One can remove
the balancing condition and consider a more general model, but this leads to
a substantial complication of the form of the operator ${\mathcal D}$
which requires a separate consideration.

The function $f(z_k,m_k)=1$ is an evident $\lambda=0$ solution
of the standard eigenvalue problem for the operator \eqref{rvD},
${\mathcal D}f(z_k,m_k)=\lambda f(z_k,m_k).$
The norm of this eigenfunction
$$
\langle 1, 1 \rangle=V_0^{(r)}(t,t_a,n,n_a;p,q)
$$
is exactly the type II multivariable analogue of the rarefied elliptic hypergeometric
function for the root system $C_d$ described above.

Another application to eigenvalue problems comes from
comparing the operator ${\mathcal D}$ with the rarefied elliptic
hypergeometric equation in the form \eqref{reheq_sym}. One can see that
the latter equation represents the eigenvalue problem for the
operator ${\mathcal D}$ with the following
three special restrictions: 1) $d=1$, 2) $t_2=t_1pq^{1-r}$, 3) $\lambda=-\nu$,
i.e. the $d=1$ function $V_{\epsilon=0}$ is now interpreted as an eigenfunction
of ${\mathcal D}$ with an additional restriction on the parameters
and a particular eigenvalue.
This is completely similar to the situation taking place in the $r=1$ case
\cite{spi:thesis,spi:tmf2007}.

\section{Conclusion}

In this paper we have proved several identities for the rarefied elliptic
hypergeometric functions and formulated a few related conjectures.
Summarizing them it is natural to
expect that all exact relations either proven \cite{bult:quadr,rai:trans,spi:umnrev}
or conjectured \cite{SV1,SV2} have rarefied analogues
obtained simply by replacing the elliptic gamma functions
$\Gamma(z;p,q)$ to $\Gamma^{(r)}(z,m;p,q)$  and integrations
$\int_{\T^d}$ to $\sum_{m_1,\ldots,m_d=0}^{r-1}\int_{\T^d}$.
If true, this yields a tremendous amount of new handbook formulae.
Furthermore, one can consider their various degeneration limits
enlarging further the number of exact formulae.
Indeed, the elliptic hypergeometric integrals can be reduced to
hyperbolic integrals \cite{JM,rai:limits}, which corresponds to the reduction of $4d$
superconformal indices to $3d$ partition functions \cite{RR2016}. Applying
a similar limit to the rarefied versions of these integrals,
one gets the rarefied hyperbolic integrals, or $3d$ partition functions
on the squashed lens space, which was mentioned already in \cite{ben}.
Some of such functions were considered recently in \cite{GahKels}.

As we have shown, the $p\to 0$ limit of the rarefied elliptic hypergeometric
functions lead to new $q$-hypergeometric identities requiring proper
systematic investigation. One can degenerate our sums of integrals to
terminating rarefied elliptic hypergeometric series and consider their
$p\to 0$ degenerations as well.

Let us shift the discrete summation variables $m_\ell \to m_\ell - [r/2]$,
where $[x]$ is the integral part of the real variable $x$, and take the limit
$r\to\infty$. Such a limit describes a degeneration of superconformal
indices of four-dimensional theories on $S^1\times L(r,-1)$ to
superconformal indices of three-dimensional field theories on the squashed
three sphere \cite{ben}. Again, this yields a very large number of exact identities
for corresponding infinite bilateral sums of $q$-hypergeometric integrals,
similar to the simplest case considered in \cite{kels}.
As shown in \cite{acm}, partition functions of $4d$ supersymmetric field
theories on $S^1\times S^3$ space-time
are equal to the corresponding superconformal indices up to an exponential
of the Casimir energy. It is natural to expect that similar situation
holds for partition functions on lens spaces and the rarefied elliptic
hypergeometric functions.

The curious discrete variable  $\epsilon$ emerging for $r>1$ deforms
various symmetries of the original elliptic hypergeometric integrals.
It breaks the $SU(n)$ (or $Sp(2n)$) gauge symmetry by
mixing with a global $U(1)$ group which acts nontrivially only on the lens space
holonomies. It plays also an important role in the definition of the
action of full $W(E_7)$ group on the vectors from the discrete space $\mathbb{Z}^8$
and related reflections acting in the space $\mathbb{Z}^{2n+2m+4}$.
So, a proper interpretation of the discrete variable $\epsilon$
in terms of the lens space superconformal indices
is one of interesting physical questions.

The integrals considered in this work are related only to the simplest lens space $L(r,-1)$.
It is possible to extend them to the general lens space $L(r,k)$, which adds more
discrete parameters. As to other applications of our results, let us mention
that it is not difficult to formulate a generalization of the
Bailey lemma \cite{spi:bailey2} on the basis on the rarefied elliptic
beta integral and use it for deriving a solution of the Yang-Baxter equation
that extends the $R$-operator of \cite{DS} to $r>1$
(this problem is related to the $2d$ integrable lattice model of \cite{kels}
and corresponding star-triangle relation). The same technique can be used
for generalization of the elliptic hypergeometric integral identities used
in the $2d$ topological field theories \cite{RR2016} together with many
other results for standard elliptic hypergeometric functions \cite{spi:umnrev}.

\medskip

{\bf Acknowledgements.}
This work is supported by the Russian Science Foundation (project no. 14-11-00598).
The author is indebted to A. P. Kels for explaining an equivalence
of the $\epsilon=0$ rarefied elliptic beta integral evaluation considered
in this paper to the one of \cite{kels}, as well as an equivalence
of the balancing conditions in  \cite{KY} involving $\mod r$ relation
to ours after appropriate adjustment of discrete parameters.
The referees are thanked for constructive remarks and suggestions.
A part of this work was done during a visit to the Max Planck
Institute for Mathematics in Bonn and the author is indebted to this
institute for a kind hospitality.

\end{document}